\documentclass[10pt] {amsart}

\usepackage{geometry}
\usepackage{amssymb}
\usepackage{amsmath}
\usepackage{amsthm}
\usepackage{amsfonts}
\usepackage{tikz}
\usetikzlibrary{matrix}
\usepackage[latin1]{inputenc}

\newcommand\restr[2]{{
  \left.\kern-\nulldelimiterspace 
  #1 
  \vphantom{\big|} 
  \right|_{#2} 
  }}

\newtheorem{theorem}{Theorem}[section]
\newtheorem{lemma}[theorem]{Lemma}
\newtheorem{proposition}[theorem]{Proposition}
\newtheorem{corollary}[theorem]{Corollary}

\theoremstyle{definition}
\newtheorem{definition}[theorem]{Definition}
\newtheorem{remark}[theorem]{Remark}
\newtheorem{example}[theorem]{Example}
\newtheorem{claim}[theorem]{Claim}

\begin{document}
\def\hii{3}
\title{Title Here}

\author{Geoffrey Scott}
\thanks{The author was partially supported by NSF RTG grant DMS-1045119}
\thanks{The author was partially supported by NSF RTG grant DMS-0943832}
%\address{Address Here}
\email{gsscott@umich.edu}

\title{The Geometry of $b^k$ Manifolds}

\begin{abstract}

Let $Z$ be a hypersurface of a manifold $M$. The $b$-tangent bundle of $(M, Z)$, whose sections are vector fields tangent to $Z$, is used to study pseudodifferential operators and stable poisson structures on $M$. In this paper we introduce the $b^k$-tangent bundle, whose sections are vector fields with ``order $k$ tangency'' to $Z$. We describe the geometry of this bundle and its dual, generalize the celebrated Mazzeo-Melrose theorem of the de Rham theory of $b$-manifolds, and apply these tools to classify certain Poisson structurse on compact oriented surfaces.
\vskip 0.5\baselineskip

\end{abstract}

\maketitle

\section{Introduction}
Melrose developed the $b$-calculus to study pseudodifferential operators on noncompact manifolds (\cite{melrose}, \cite{grieser}). Considering the manifold in question as the interior of a manifold $M$ with boundary, he constructed the {\it $b$-tangent bundle} ${^{b}}TM$ whose sections are vector fields on $M$ tangent to $\partial M$, and the {\it $b$-cotangent bundle} ${^{b}}T^*M$, whose sections are differential forms with a specific kind of order-one singularity at $\partial M$. The authors of \cite{gmp2} have recently applied these ideas to study global Poisson geometry; in this context, ${^b}TM$ and ${^b}T^*M$ are defined on a manifold $M$ with a distinguished hypersurface $Z$ rather than on a manifold with boundary\footnote{These competing perspectives can be reconciled by viewing a manifold with boundary $M$ as one half of its double. In doing so, the boundary of $M$ corresponds to a hypersurface of the double. In this paper, we follow the precedent of \cite{gmp2} and define our bundles over manifolds with distinguished hypersurfaces.}, and sections of ${^b}TM$ (and ${^b}T^*M$) are vector fields (and differential forms) tangent to $Z$ (or singular at $Z$). In this paper, we generalize this construction so that vector fields and differential forms with higher order tangency and higher order singularity may also be realized as sections of bundles.

The construction of these bundles in Section 2 is subtle: although we wish to define a $b^k$-vector field as a vector field with an ``order $k$ tangency to $Z$,'' there is no straightforward way to rigorously define this notion. To do so, we must include in the definition of a $b^k$-manifold the data of a $(k-1)$-jet of $Z$ (and insist that the morphisms in the $b^k$-category preserve this jet). We then define a $b^k$-vector field as a vector field $v$ such that $\mathcal{L}_{v}(f)$ vanishes to order $k$ for functions $f$ that represent the jet data. Then we define the $b^k$-tangent bundle ${^{b^k}}\hspace{-1pt}TM$ as the vector bundle whose sections are $b^k$-vector fields, and the $b^k$-cotangent bundle ${^{b^k}}\hspace{-1pt}T^*M$ as its dual. When $k = 1$, these are the familiar constructions from \cite{melrose} and \cite{gmp2}.

In Section 3 we study the geometry of the fibers of ${^{b^k}}\hspace{-1pt}TM$ and ${^{b^k}}\hspace{-1pt}T^*M$. Recall from \cite{gmp2} that the fibers of ${^b}TM$ and ${^b}T^*M$ satisfy 
\[
{^{b}}T_pM \cong \left\{ \begin{array}{c l}
T_pM  &\textrm{for} \ p \notin Z\\
T_pZ + \langle y \frac{\partial}{\partial y} \rangle   &\textrm{for} \ p \in Z \end{array}\right.
\textrm{ \ \  \ \ }
{^{b}}T_p^*M \cong \left\{ \begin{array}{c l}
T_p^*M & \textrm{for} \ p \notin Z\\
T_p^*Z + \langle \frac{dy}{y} \rangle  & \textrm{for} \ p \in Z \end{array}\right.
\]
\noindent where $y$ is a defining function for $Z$. Similarly, we show that the fibers of ${^{b^k}}TM$ and ${^{b^k}}T^*M$ satisfy 
\[
{^{b^k}}T_pM \cong \left\{ \begin{array}{c l}
T_pM & \textrm{for} \ p \notin Z\\
T_pZ + \langle y^k \frac{\partial}{\partial y} \rangle  & \textrm{for} \ p \in Z \end{array}\right.
\textrm{ \ \ \ \ }
{^{b^k}}T_p^*M \cong \left\{ \begin{array}{c l}
T_p^*M & \textrm{for} \ p \notin Z\\
T_p^*Z + \langle \frac{dy}{y^k} \rangle  & \textrm{for} \ p \in Z \end{array}\right.
\]
where $y$ is a local defining function for $Z$ that represents the jet data of the $b^k$-manifold.

In Section 4 we define a differential on the complex of {\it $b^k$-forms} (sections of the exterior algebra of ${^{b^k}}\hspace{-1pt}T^*M$) and prove a Mazzeo-Melrose type theorem for the cohomology ${^{b^k}}H^*(M)$ of this complex. 
\begin{equation}\label{nci}
{^{b^k}}\hspace{-1pt}H^p(M) \cong H^p(M) \oplus \left( H^{p-1}(Z) \right)^k
\end{equation}
However, this isomorphism (like that of the classic Mazzeo-Melrose theorem) is non-canonical. By defining the {\it Laurent Series} of a $b^k$-form, which expresses a $b^k$-form as a sum of simpler $b^{\ell}$-forms (for $\ell \leq k$), we show that there {\it is} a way to construct the isomorphism in Equation \ref{nci} so that the $\left( H^{p-1}(Z) \right)^k$ summand of a $b^k$-cohomology class {\it is} canonically defined.

In Section 5, we study the geometry of $b^k$-forms of top degree. In \cite{radko}, the author defined the {\it Liouville volume} of a $b$-form of top degree as a certain principal value of the form. This invariant was featured in her classification theorem of stable Poisson structures on compact surfaces. We generalize this by defining the {\it volume polynomial} of a $b^k$-form of top degree. This polynomial encodes the asymptotic behavior of the integral of the form near $Z$. We define the {\it Liouville volume} as the constant term of this polynomial -- it agrees with the classic definition of Liouville volume when $k = 1$. We can also take the Liouville volume of a degree $p$ $b^k$-form along any $p$ dimensional submanifold of $M$. Citing Poincar\'{e} duality, we define the {\it smooth part} of a $b^k$ cohomology class $[\omega]$ to be the de Rham cohomology class whose integrals along $p$-cycles equal the Liouville volumes of $\omega$ along these cycles. At the end of Section 5, we use these tools to realize the abstract isomorphism in Equation \ref{nci} with an explicit canonical map. The image of a $b^k$ form under this map is its {\it Liouville-Laurent decomposition}.

In Section 6, we define a symplectic $b^k$-form as a closed $b^k$ 2-form having full rank (when $k = 1$, these are frequently called {\it log symplectic} forms), and prove the classic Moser theorems in the $b^k$ category. We also revisit the classification theorems of stable Poisson structures on compact oriented surfaces from \cite{radko} and  \cite{gmp2}. Radko classifies stable Poisson structures using geometric data, while the authors of \cite{gmp2} use cohomological data; in this paper, we show how the Liouville-Laurent decomposition relates the geometric data to the cohomological data.

This paper ends in Section 7 with an example of how the theory of $b^k$-manifolds can answer questions from outside $b^k$-geometry. Let $\Pi$ be a Poisson structure on a manifold $M$ whose rank differs from $\textnormal{dim}(M)$ precisely on a hypersurface $Z$. We say that $\Pi$ is {\it of $b^k$-type} if it is dual to a symplectic $b^k$-form for {\it some} choice of jet data (respectively, we say that the symplectic form on $M \backslash Z$ is {\it of $b^k$-type}). When $M$ is a surface, this means that the Poisson bivector $\Pi$ is given by $f\Pi_0$ where $\Pi_0$ is dual to a symplectic form, and $f$ is locally the $k^{\textrm{th}}$ power of a defining function of $Z$. We give a condition for two such Poisson structures on a compact surface to be isomorphic in terms of the summands in their respective Liouville-Laurent decompositions.

{\bf Acknowledgements:} I thank both Daniel Burns and Victor Guillemin for helpful discussions, and the latter for introducing me to Melrose's $b$-calculus. Also, to the staff of 1369 Coffee House who tolerated my near-residence of their establishment while preparing this paper -- thank you.

\section{Preliminaries}

In this section, we establish notation pertaining to jet bundles, review definitions from the theory of $b$-manifolds, and generalize these definitions. All manifolds, maps, and vector fields are assumed to be smooth.

\subsection{Notation}
Let $i: Z \rightarrow M$ be the inclusion of a hypersurface into a manifold, let $C^{\infty}$ be the sheaf of smooth functions on $M$, and let $\mathcal{I}_Z \subseteq C^{\infty}$ be the ideal sheaf of $Z$.
\begin{definition} The {\bf sheaf of germs} at $Z$ is $i^{-1}(C^{\infty})$; a {\bf germ} at $Z$ is a global section of this sheaf. The {\bf sheaf of $k$-jets} at $Z$ is $\mathcal{J}_Z^k := i^{-1}(C^{\infty} / \mathcal{I}_Z^{k+1})$; a {\bf $k$-jet} at $Z$ is a global section of this sheaf.
\end{definition}

We will write $J^k_{Z}$ (or simply $J^k$) to denote the $k$-jets at $Z$, and $I_Z$ (or simply $I$) to denote the global sections of $i^{-1}(\mathcal{I}_Z)$. We write $[f]_Z^k$ (or simply $[f]^k$) to denote the $k$-jet represented by a smooth function $f$ defined in a neighborhood of $Z$. Also, if $j$ is a $k$-jet, we write $f \in j$ if $f$ represents $j$ and $f \in I^k$ if $f$ represents an element of $I^k$ (equivalently, if $[f]^{k-1} = 0$).

\subsection{Definitions}
In \cite{gmp2}, the authors use $b$-manifolds to study symplectic forms having order-one singularities along a hypersurface. In this paper, we generalize these techniques to study symplectic forms having higher-order singularities. We begin by recalling the basic definitions from the theory of $b$-manifolds. See \cite{gmp2} for more exposition.
\begin{definition} A {\bf $b$-manifold} is a pair $(M, Z)$ of a smooth oriented manifold $M$ and an oriented hypersurface $Z \subseteq M$ such that $Z = \{f = 0\}$ for some global defining function $f: M \rightarrow \mathbb{R}$.
\end{definition}
\begin{definition} A {\bf $b$-map} from $(M, Z)$ to $(M', Z')$ is a map $\varphi: M \rightarrow M'$ such that $\varphi^{-1}(Z') = Z$ and $\varphi$ is transverse to $Z'$.
\end{definition}
\begin{definition} A {\bf $b$-vector field} on $(M, Z)$ is a vector field $v$ on $M$ such that $v_p \in T_pZ$ for all $p \in Z$.
\end{definition}
\begin{definition} The {\bf $b$-tangent bundle} $^bTM$ on $(M, Z)$ is the vector bundle whose sections are the $b$-vector fields on $(M, Z)$.
\end{definition}
\begin{definition} The {\bf $b$-cotangent bundle} $^bT^*M$ is the dual bundle of $^bTM$.
\end{definition}

The authors of \cite{gmp2} show that sections of the exterior algebra of the $b$-cotangent bundle are differential forms on $M$ with a certain kind of order-one singularity at $Z$. Towards the goal of constructing similar bundles to study differential forms with higher-order singularities, we wish to define a $b^k$-vector field as a vector field ``tangent to order $k$ on $Z$.'' However, the next example shows that the na\"ive definition of being ``tangent to order $k$ on $Z$'' (as a vector field $v$ such that $\mathcal{L}_{v}(f) \in I^k$ for a defining function $f$ of $Z$) is ill-defined.

\begin{example} 
On the $b$-manifold $(M, Z) = (\{(x, y) \in \mathbb{R}^2\}, \{y = 0\})$, two different defining functions for $Z$ are given by $y$ and $e^xy$. The vector field $v = \frac{\partial}{\partial x}$ satisfies 
\[
\mathcal{L}_{v}(y) = 0 \in I^2
\hspace{1cm}
\textrm{and}
\hspace{1cm}
\mathcal{L}_{v}(e^xy) = e^xy \notin I^2
\]
so the order of vanishing of the Lie derivative of a defining function depends on the choice of defining function. 
\end{example}
This phenomenon prevents us from emulating the \cite{gmp2} paper mutatis mutandis; we must endow our $b$-manifolds with additional data to make possible the definition of a $b^k$-vector field.

\begin{definition} For $k \geq 1$, a {\bf $b^k$-manifold} is a triple $(M, Z, j_Z)$ where 
\begin{itemize}
\item $M$ is an oriented manifold.
\item $Z \subseteq M$ is an oriented hypersurface.
\item $j_Z$ is an element of $J^{k-1}_Z$ that can be represented by a positively oriented local defining function $y$ for $Z$ (that is, if $\Omega_Z$ is a positively oriented volume form of $Z$, then $dy \wedge \Omega_Z$ is positively oriented for $M$)
\end{itemize}
\end{definition}

If $k > 1$ and a function shares the same $(k-1)$-jet as a positively oriented local defining function for $Z$, then it itself is a positively oriented local defining function for $Z$. In this case, any $f \in j_{Z}$ is a positively oriented local defining function for $Z$. When $k = 1$ the jet data $\{j_{Z}\}$ is vacuous (because any local defining function for $Z$ represents the trivial $0$-jet), so the definition of a $b^1$-manifold nearly\footnote{We do not demand that $Z = \{f = 0\}$ for some {\it globally}-defined $f$, so the definition of a $b^1$-manifold is slightly more general than the definition of a $b$-manifold given in \cite{gmp2}. However, any symplectic $b^1$-manifold {\em will} have the property that $Z$ is defined by a global function, so the symplectic geometry of $b^1$-manifolds coincides with the symplectic geometry of $b$-manifolds.} agrees with that of a $b$-manifold. 
\begin{definition}
A {\bf $b^k$-map} from $(M, Z, j_{Z})$ to $(M', Z', j_{Z'})$ is a map $\varphi: M \rightarrow M'$ such that $\varphi^{-1}(Z') = Z$, $\varphi$ is transverse to $Z'$, and $\varphi^*(j_{Z'}) = j_{Z}$.
\end{definition}

The interested reader is invited to check that $b^k$-manifolds and $b^k$-maps form a category.

\begin{remark}\label{this_remark} Given a hypersurface $Z \subseteq M$, a vector field $v$ on $M$ with $v_p \in T_pZ$ for all $p \in Z$, and $f \in C^{\infty}(M)$, the jet $[\mathcal{L}_v(f)]^{k-1}$ depends only on $[f]^{k-1}$.
\end{remark}
\begin{proof}
If $[{f}_2]^{k-1} = [f_1]^{k-1}$, then ${f}_2 - f_1 = y^{k}g$ for a local defining function $y$ and some smooth $g$. For a vector field $v$ satisfying $v_p \in T_pZ$, $\mathcal{L}_v(y) \in I$, so
\begin{align*}
[\mathcal{L}_v({f}_2)]^{k-1} &= [\mathcal{L}_v(f_1) + y^{k}\mathcal{L}_v(g) + kgy^{k-1}\mathcal{L}_v(y)]^{k-1}\\
&= [\mathcal{L}_v(f_1)]^{k-1}\textrm{.}
\end{align*}
\end{proof}

Remark \ref{this_remark} shows that the following definition makes sense.

\begin{definition}\label{bk_vectorfield} A {\bf $b^k$-vector field} on $(M, Z, j_Z)$ is a vector field $v$ with $v_p \in T_p(Z)$ for $p \in Z$ such that for any $f \in j_{Z}$, $\mathcal{L}_v(f) \in I^k$.
\end{definition}

To check whether a vector field $v$ is a $b^k$-vector field, it suffices (by Remark \ref{this_remark}) to check that $\mathcal{L}_v(f) \in I^k$ for just one local defining function $f \in j_Z$. The following example shows that Definition \ref{bk_vectorfield} formalizes the notion of a vector field being ``tangent to order $k$'' along a hypersurface.

\begin{example}\label{the_above_example} On the $b^k$ manifold $(\mathbb{R}^n, Z = \{x_n = 0\}, [x_n]^{k-1})$, a vector field $v = \sum_{i = 1}^{n} v_i \frac{\partial}{\partial x_i}$ is a $b^k$-vector field iff
\[
\mathcal{L}_v(x_n)\in I^k
\]
which occurs iff $v_{n} \in I^k$. That is, the $b^k$-vector fields are precisely those of the form
\[
\phi_n x_n^k \frac{\partial}{\partial x_n} + \sum_{i = 1}^{n-1} \phi_i \frac{\partial}{\partial x_i}
\]
for smooth functions $\phi_i$.
\end{example}

On a $b^{k}$-manifold $(M, Z, j_{Z})$, each $p \notin Z$ is contained in a coordinate neighborhood $(U, \{x_1, \dots, x_n\})$ on which $\{\frac{\partial}{\partial x_i}\}$ generate the space of $b^k$-vector fields over $U$ as a free $C^{\infty}(U)$-module. For points $p \in Z$, Example \ref{the_above_example} shows that on a coordinate neighborhood $(U, \{x_1, \dots, x_n\})$ of $p$ with $x_n \in j_Z$, the vector fields
\[
\left\{\frac{\partial}{\partial x_1}, \dots, \frac{\partial}{\partial x_{n-1}}, x_n^k \frac{\partial}{\partial x_n}\right\}
\]
generate the space of ${b}^k$-vector fields over $U$ as a $C^{\infty}(U)$-module. Consequently, $b^k$-vector fields form a projective $C^{\infty}$ module over $M$, as well as a Lie subalgebra of the algebra of vector fields on $M$, so we can realize $b^k$-vector fields as the sections of a bundle on $M$.

We call this bundle $^{{b}^k}TM$ the {\bf $b^k$-tangent bundle}. The dual of this bundle $^{{b}^k}T^*M$ is the {\bf $b^k$-cotangent bundle}. When $k = 1$ we recover the classic definitions of a $b$-vector field and the $b$-(co)tangent bundle. We write ${^{b^k}}\hspace{-2pt}\Omega^p(M)$  for sections of $\wedge^{p} ({^{b^k}}T^*M)$. Elements of ${^{b^k}}\hspace{-2pt}\Omega^p(M)$ are {\bf differential $b^k$-forms}.

\section{Geometry of the $b^k$-(co)tangent bundle}

In this section, we describe the fibers of the $b^k$-(co)tangent bundles and study maps between $b^k$-(co)tangent bundles as $k$ varies. These results will prepare us to study the de Rham theory and the symplectic geometry of $b^k$-manifolds.

Let ${^{b^k}}\textrm{Vect}(M)$ be the space of $b^k$-vector fields and $C_p^{\infty}(M)$ be the ideal of functions vanishing at $p \in M$. We can define ${^{b^k}}T_pM$ intrinsically as
\[
{^{b^k}}T_pM \cong {^{b^k}}\textnormal{Vect}(M) / (C_p^{\infty}(M) {^{b^k}}\textnormal{Vect}(M)).
\]
There is a canonical map that relates the fibers of ${^{b^k}}TM$ to those of $TM$.
\begin{equation}\label{vectorspacemap}
{^{b^k}}T_pM \cong \frac{{^{b^k}}\textnormal{Vect}(M)}{C_p^{\infty}(M) {^{b^k}}\textnormal{Vect}(M)} \rightarrow \frac{\textnormal{Vect}(M)}{C_p^{\infty}(M)\textnormal{Vect}(M)} \cong T_pM
\end{equation}
The results of this section will show that for $p \in Z$, there is a canonical element in the kernel of Map \ref{vectorspacemap} (and dually a canonical element in the quotient ${^{b^k}}T^*_pM / T^*_pM$). Instead of proving these results using this intrinsic description of individual fibers, we will take a more global perspective in order to more closely follow the exposition and results of \cite{gmp2}.

\subsection{Fibers of the $b^k$-(co)tangent Bundle}
Similar to the $b$-manifold case, there are maps between the (co)tangent bundles of $Z$ and the $b^k$-(co)tangent bundles of $M$ restricted to $Z$.
\begin{align}\label{tgt_map}
^{{b}^k}T\restr{M}{Z} &\twoheadrightarrow TZ\\
\label{ctgt_map} ^{{b}^k}T^*\restr{M}{Z} &\hookleftarrow T^*Z
\end{align}

Map \ref{tgt_map} is induced by the map of sections $\Gamma(M, ^{{b}^k}\hspace{-2pt}TM) \rightarrow \Gamma(Z, TZ)$ given by restricting a ${{b}^k}$-vector field to $Z$. Map \ref{ctgt_map} is dual to Map \ref{tgt_map}. We study the (co)kernel of these maps, starting with a technical remark.

\begin{remark}\label{tech_rmk} Let $v$ be a $b^k$-vector field that vanishes on $Z$ when viewed as a section of $TM$, and let $x_n \in j_Z$ be a local defining function for $Z$. Then $v$ also vanishes on $Z$ as a section of ${^{b^k}}TM$ precisely at those points where the $k$-jet $[\mathcal{L}_v(x_n)]^k$ vanishes.
\end{remark}
\begin{proof} In local coordinates $\{x_1, \dots, x_n\}$
\[
v = \phi_n x_n^k \frac{\partial}{\partial x_n} + \sum_{i < n} \phi_i \frac{\partial}{\partial x_i}
\]
where $\{\phi_{i}\}_{i \leq n}$ are smooth functions and $\{\phi_{i}\}_{i < n}$ vanish on $Z$. Because the functions $\{\phi_{i}\}_{i \leq n}$ constitute the trivialization of ${^{b^k}}TM$ induced by the local coordinates, $v$ vanishes on $Z$ as a section of ${^{b^k}}TM$ at those points of $Z$ where $\phi_n$ vanishes, which are precisely the points where $[\mathcal{L}_v(x_n)]^k = [\phi_nx_n^k]^k$ vanishes.
\end{proof}
\begin{proposition}\label{cnvs}
The kernel of Map \ref{tgt_map} has a canonical nowhere vanishing section.
\end{proposition}
\begin{proof} Pick a local defining function $y \in j_{Z}$ and a vector field $v$ satisfying $\restr{dy(v)}{Z} = 1$. Then $[\mathcal{L}_{y^{k}v}(y)]^k = [y^{k}\mathcal{L}_v(y)]^k$ is nonvanishing. By Remark \ref{tech_rmk}, $y^{k}v$ is a ${b}^k$-vector that vanishes on $Z$ as a section of $TM$ but is nowhere vanishing as a section of ${^{b^k}}TM$.

To prove that $y^kv$ is canonical, suppose $y_2 \in j_Z$ and $v_2$ are different choices of defining function and vector field. Then $y_2 = y(1 + gy^{k-1})$ for some smooth $g$ and
\begin{align*}
\mathcal{L}_{v_2}(y_2) &= \mathcal{L}_{v_2}(y) + gy^{k-1}\mathcal{L}_{v_2}(y)+ y\mathcal{L}_{v_2}(gy^{k-1})\\
&= (1 + kgy^{k-1})\mathcal{L}_{v_2}(y) + y^k\mathcal{L}_{v_2}(g)
\end{align*}
and
\begin{align*}
[\mathcal{L}_{y^kv - y_2^kv_2}&(y)]^{k} = [y^k \mathcal{L}_v(y) - y^k(1 + gy^{k-1})^k\mathcal{L}_{v_2}(y)]^k\\
&= \left[y^k \mathcal{L}_v(y) - y^k(1 + gy^{k-1})^k\frac{\mathcal{L}_{v_2}(y_2) - y^{k}\mathcal{L}_{v_2}(g)}{1 + kgy^{k-1}}\right]^k\\
&= 0\\
\end{align*}
By Remark \ref{tech_rmk}, $y^kv - y_2^kv_2$ vanishes on $Z$ as a section of ${^{b^k}}TM$, so $y^kv$ and $y_2^kv_2$ represent the same section of ${^{b^k}}T\restr{M}{Z}$.
\end{proof}

\noindent Turning our attention to the cotangent bundle, observe that although the differential form $y^{-k}dy$ is not defined on $Z$ as a section of $T^*M$, its pairing with any $b^k$-vector field extends smoothly over $Z$. Therefore, $y^{-k}dy$ extends smoothly over $Z$ as a section of ${^{b^k}}T^*M$. By pairing $y^{-k}dy$ with a representative of a nonvanishing section of ker($^{{b}^k}T\restr{M}{Z} \rightarrow TZ$), we see that $y^{-k}dy$ is nonwhere vanishing. This proves the following claim.

\begin{claim} The cokernel of Map \ref{ctgt_map} has a nowhere vanishing section.
\end{claim}
The preceding discussion describes of the fibers of the $b^k$-(co)tangent bundle of a $b^k$-manifold $(M, Z, [y]^{k-1})$ as follows.
\begin{align*}
{^{b^k}}T_pM &\cong \left\{ \begin{array}{c l}
T_pM & \textrm{for} \ p \notin Z\\
T_pZ + \langle y^k \frac{\partial}{\partial y} \rangle  & \textrm{for} \ p \in Z \end{array}\right.\\
{^{b^k}}T_p^*M &\cong \left\{ \begin{array}{c l}
T_p^*M & \textrm{for} \ p \notin Z\\
T_p^*Z + \langle \frac{dy}{y^k} \rangle  & \textrm{for} \ p \in Z \end{array}\right.
\end{align*}
\subsection{Properties of $b^k$-Forms}
From the above description of the fibers of ${^{b^k}}T_pM$, we see that $\Omega^p(M \backslash Z) \cong {^{b^k}}\Omega^p(M \backslash Z)$. That is, every $b^k$-form restricts to an ordinary differential form on $M \backslash Z$. We can therefore interpret a $b^k$-form as a differential form on $M \backslash Z$ that satisfies certain asymptotic properties (prescribed by the jet data) around $Z$. We also see that for any defining function $y \in j_Z$, every $b^k$-form can be written in a neighborhood $U$ of $Z$ in the form 
\begin{equation}\label{basic_decomp}
\omega = \frac{dy}{y^k} \wedge \alpha + \beta
\end{equation}
for differential forms $\alpha \in \Omega^{p-1}(U)$ and $\beta \in \Omega^{p}(U)$. Although the forms $\alpha$ and $\beta$ appearing in Equation \ref{basic_decomp} are not uniquely defined by $\omega$, we will show that $i^*(\alpha)$ is independent of the choice of $y, \alpha$ and $\beta$, where $i: Z \rightarrow M$ is the inclusion.

\begin{proposition}\label{f1f2} On a $b^k$-manifold, if $f_1, f_2 \in j_Z$ are local defining functions for $Z$, then in a neighborhood $U$ of $Z$
\[
\frac{df_1}{f_1^k} = \frac{df_2}{f_2^k} + \beta
\]
where $\beta \in \Omega^1(U)$.
\end{proposition}
\begin{proof} The proof is technical. See Section \ref{yucky_section} for the details.
\end{proof}

\begin{corollary}\label{indepedence_corollary} Given a decomposition of $\omega \in {^{b^k}}\hspace{-2pt}\Omega(M)$ as in Equation \ref{basic_decomp}, $i^*(\alpha)$ is independent of the decomposition.
\end{corollary}
\begin{proof} Let $\alpha_1$ and $\alpha_2$ be the $\alpha$ terms of two such decompositions. Setting the decompositions equal and applying the preceding proposition shows that
\[
\frac{dy}{y^k}\wedge(\alpha_2 - \alpha_1)
\]
is a smooth form for some local defining function $y \in j_Z$, so $i^*(\alpha_2 - \alpha_1) = 0$.
\end{proof}
This proves the well-definedness of the map 
\begin{align}\label{alphamap}
\iota_{\mathbb{L}}: {^{b^k}}\hspace{-2pt}\Omega^p(M) &\rightarrow \Omega^{p-1}(Z)\\
\frac{dy}{y^k} \wedge \alpha + \beta &\mapsto i^*(\alpha) \notag
\end{align}
Alternatively, this map can be defined by restricting a form to $Z$, then contracting with the canonical section $\mathbb{L}$ described in Proposition \ref{cnvs}. This motivates the notation $\iota_{\mathbb{L}}$ for the map.

Equation \ref{basic_decomp} might give us hope that we can define a $b^k$-form without reference to any jet data as ``a form $\omega$ on $M \backslash Z$ which admits a decomposition $\omega = y^{-k}dy \wedge \alpha + \beta$ in a neighborhood of $Z$ for {\it some} local defining function $y$''. However, for a fixed $\omega$ the existence of a decomposition $\omega = y^{-k}dy \wedge \alpha + \beta$ depends strongly on $[y]^{k-1}$. It turns out that the set of $\omega \in \Omega(M \backslash Z)$ which extends over $Z$ with respect to {\em some} $[y]^{k-1}$ is not even closed under addition. This hopefully motivates (for a second time) the necessity of the jet data in the definition of a $b^k$-manifold.

\subsection{Viewing a $b^{\ell}$-Form as a $b^k$-Form}
To prepare for the next section, we consider a new family of maps between the $b^k$-(co)tangent bundles. These maps generalize the fact that any differential form is naturally a $b^k$-form.

For any $0 < \ell \leq k$, the natural map $J_Z^{k-1} \rightarrow J_Z^{\ell - 1}$ allows us to canonically endow a $b^k$-manifold $(M, Z, j_Z)$ with a $b^{\ell}$-manifold structure. Defining ${^{b^0}}TM := TM$ and ${^{b^0}}T^*M := T^*M$ for notational convenience, a $b^k$-manifold structure on $M$ defines $2k+2$ different bundles $^{b^{\ell}}TM$, $^{b^{\ell}}T^*M$ over $M$ for $0 \leq \ell \leq k$. A $b^k$-vector field will also be a $b^{\ell}$-vector field for the induced $b^{\ell}$-manifold structure. This induces a map
\begin{equation}\label{vf_downgrade}
^{{b}^{k}}TM \rightarrow ^{{b}^{\ell}}TM
\end{equation}
and its dual
\begin{equation}\label{df_upgrade}
^{{b}^{\ell}}T^*M \rightarrow ^{{b}^{k}}T^*M\textrm{,}\\
\end{equation}
the latter of which can be described explicitly in terms of the decompositions from Equation \ref{basic_decomp} as
\[
\frac{dy}{y^{\ell}} \wedge \alpha + \beta \mapsto \frac{dy}{y^k} \wedge (y^{k-\ell} \alpha) + \beta\textrm{.}
\]

\section{De Rham Theory and Laurent Series of $b^k$-forms}\label{sec_derham}

We define a differential $d: {^{b^k}}\hspace{-2pt}\Omega^p(M) \rightarrow {^{b^k}}\hspace{-2pt}\Omega^{p+1}(M)$ by
\[
d \left( \frac{dy}{y^k} \wedge \alpha + \beta \right) = \frac{dy}{y^k} \wedge d\alpha + d\beta \textrm{.}
\]
This definition does not depend on the decomposition. Indeed, $d(\omega)$ is the unique extension of the image of the classic de Rham differential $d(\restr{\omega}{M \backslash Z}) \in \Omega^p(M \backslash Z) \cong {^{b^k}}\hspace{-2pt}\Omega^p(M \backslash Z)$ over $Z$. 

\begin{definition} The {\bf $b^k$-de Rham complex} is $( {^{b^k}}\hspace{-2pt}\Omega^p(M), d)$, with ${^{b^k}}\hspace{-2pt}\Omega^0(M) := C^{\infty}(M)$. The {\bf $b^k$-cohomology} ${^{b^k}}\hspace{-2pt}H^*(M)$ is the cohomology of this complex.
\end{definition}

\begin{proposition}\label{ses_exact} The sequence below, with $g$ given by Map $\eqref{df_upgrade}$, is exact
\begin{equation}\label{ses1} 
0 \rightarrow ^{b^{k-1}}\hspace{-2pt}\Omega^p(M) \stackrel{g}{\rightarrow} {^{b^k}}\hspace{-2pt}\Omega^p(M) \stackrel{\iota_{\mathbb{L}}}{\rightarrow} \Omega^{p-1}(Z) \rightarrow 0\textrm{.}
\end{equation}
Moreover, for any closed $\alpha \in \Omega^{p-1}(Z)$ and collar neighborhood $(y, \pi): U \rightarrow \mathbb{R} \times Z$ of $Z$ with $y \in j_Z$, there is a closed form $\omega \in \iota_{\mathbb{L}}^{-1}(\alpha)$ such that 
\[
\omega = \frac{dy}{y^k} \wedge \pi^*(\alpha)
\]
in a neighborhood of $Z$.
\end{proposition}
\begin{proof} The only nontrivial part of the exactness claim is that $\textrm{ker}(\iota_{\mathbb{L}}) \subseteq \textrm{im}(g)$. The kernel of $\iota_{\mathbb{L}}$ consists precisely of those $\omega$ that admit some decomposition 
\[
\omega = \frac{dy}{y^k} \wedge \alpha + \beta
\]
in a neighborhood of $Z$ for which $i^*(\alpha) = 0$. Locally around $Z$, $T^*M$ splits as $T^*Z + \langle dy \rangle$, so we may replace $\alpha$ by a form that vanishes on $Z$ without changing $\omega$. Then $y^{-1}\alpha$ is a smooth form, and 
\[
\frac{dy}{y^{k-1}} \wedge \frac{\alpha}{y} + \beta
\]
extends over $M$ to a $b^{k-1}$ form in $g^{-1}(\omega)$. Therefore, Sequence \ref{ses1} is exact.

Given a closed $\alpha \in \Omega^{p-1}(Z)$ and a collar neighborhood $(y, \pi): U \rightarrow (-R, R) \times Z$ of $Z$ with $y \in j_Z$, let $\widetilde{y} \in C^{\infty}(M)$ be a function that agrees with $y$ on $(-R/2, R/2) \times Z$ and is locally constant outside $U$. Then the $b^k$-form $\omega = \widetilde{y}^{-k}d\widetilde{y} \wedge \pi^*(\alpha)$ extends to a closed $b^k$-form on $M$ that vanishes outside $U$ and satisfies $\iota_{\mathbb{L}}(\omega) = \alpha$. In $(-R/2, R/2) \times Z$,
\[
\omega = \frac{dy}{y^k}\wedge \pi^*(\alpha)
\]
and $d\pi^*(\alpha) = \pi^*(d\alpha) = 0$.
\end{proof}
One can check that the short exact sequence from Proposition \ref{ses_exact} is a chain map of complexes, hence induces a long exact sequence
\[
\dots \rightarrow {^{b^{k-1}}}\hspace{-2pt}H^*(M) \rightarrow {^{b^k}}\hspace{-2pt}H^*(M) \rightarrow H^{*-1}(Z) \rightarrow {^{b^{k-1}}}\hspace{-2pt}H^{*+1}(M) \rightarrow \dots
\]
By Proposition \ref{ses_exact}, the maps ${^{b^k}}\hspace{-2pt}H^*(M) \rightarrow H^{*-1}(Z)$ are surjective, so the long exact sequence is a collection of short exact sequences
\begin{align}\label{chly_ses}
0 \rightarrow  {^{b^{k-1}}}\hspace{-2pt}H^p(M) \rightarrow  {^{b^k}}\hspace{-2pt}H^p(M) \rightarrow H^{p-1}(Z)\rightarrow 0
\end{align}
Using induction on $k$, this proves the following proposition.
\begin{proposition}\label{noncanonicaliso}
\[
{^{b^k}}\hspace{-2pt}H^p(M) \cong H^p(M) \oplus \left( H^{p-1}(Z) \right)^k
\]
\end{proposition}
\begin{proof} From the remarks above.
\end{proof}

So far, this isomorphism is non-canonical: although we can lift every $[\alpha] \in H^{p-1}(Z)$ in Equation \ref{chly_ses} to an element of ${^{b^k}}H^{p}(M)$, we do not yet have a preferred choice of lifting, and different choices yield genuinely different isomorphisms. Results in Subsection \ref{subsection_laurent}, where we show that the $(H^{p-1}(Z))^k$ summand of the image of any $[\omega] \in {^{b^k}}\hspace{-2pt}H^p(M)$ {\it can} be canonically defined, will give us partial relief from this uncomfortable state of affairs. Finally, in Section \ref{section_volume} we will give an explicit canonical map for the isomorphism in Proposition \ref{noncanonicaliso}, and in doing so we will see a geometric interpretation for the terms on the right side of the isomorphism.

\subsection{The Laurent Series of a Closed $b^k$-Form}\label{subsection_laurent}

\begin{definition} A {\bf Laurent Series} of a closed $b^k$-form $\omega$ is an expression for $\omega$ in a neighborhood of $Z$ of the form
\[
\omega = \sum_{i = 1}^k \frac{dy}{y^i} \wedge \alpha_{-i} + \beta
\]
where $y \in j_Z$ is a positively oriented local defining function and each $\alpha_{-i}$ is closed.
\end{definition}

\begin{remark}\label{proj_gives_ld} Every closed $b^k$-form has a Laurent series. In fact, Proposition \ref{ses_exact} shows that given a collar neighborhood $(y, \pi): U \rightarrow (-R, R) \times Z$ of $Z$ with $y \in j_Z$, every closed $b^k$-form $\omega$ can be written (in a neighborhood of $Z$) as the sum of a closed $b^{k-1}$ form and
\[
\frac{dy}{y^k}\wedge \pi^*(\iota_{\mathbb{L}}\omega)\textrm{.}
\]
By applying induction on the $b^{k-1}$ form, we arrive at a Laurent series of the form 
\[
\omega = \sum_{i = 1}^k \frac{dy}{y^i} \wedge \pi^*(\gamma_{-i}) + \beta
\]
for closed forms $\gamma_{-i}$ on $Z$.
\end{remark}
\begin{example}
Consider the $b^k$-manifold $(S^1 \times S^1, Z_1 \cup Z_2, [y]^{k-1})$ pictured in Figure \ref{fig_disconnect}. 

\begin{figure}[!ht]
\centering
\begin{tikzpicture}[scale = 0.6]

%fills
\draw[fill = gray, dashed] (0, 0) -- (0, 3) arc (180:360:1 and 0.5) -- (2, 0) arc (0:-180:1 and 0.5);
\draw[fill = lightgray, dashed] (2, 3) arc (0:360:1 and 0.5);

\draw[fill = gray, dashed, shift = {(4, 0)}] (0, 0) -- (0, 3) arc (180:360:1 and 0.5) -- (2, 0) arc (0:-180:1 and 0.5);
\draw[fill = lightgray, dashed, shift = {(4, 0)}] (2, 3) arc (0:360:1 and 0.5);

%cylinder walls
\draw[thick] (0, 0) -- (0, 3);
\draw[thick] (2, 0) -- (2, 3) node[below right]{$U_1$};
\draw[thick, shift = {(4,0)}] (0, 0) -- (0, 3);
\draw[thick, shift = {(4,0)}] (2, 0) -- (2, 3) node[below right]{$U_2$};

%midlines
\draw[ultra thick] (0, 1.5) node[left]{$Z_1$} arc (180:360:1 and 0.5);
\draw[ultra thick, dashed] (0, 1.5) arc (180:0:1 and 0.5);

\draw[ultra thick, shift = {(4, 0)}] (0, 1.5) node[left]{$Z_2$} arc (180:360:1 and 0.5);
\draw[ultra thick, dashed, shift = {(4, 0)}] (0, 1.5) arc (180:0:1 and 0.5);

%toplines
\draw[thick] (2, 3) arc (180:90:1.15 and 1.3);
\draw[thick] (4, 3) arc (0:97:1.15 and 1.35);

\draw[thick] (0, 3) arc (180:0:3 and 3);

\draw[thick] (2, 0) arc (180:270:1.15 and 1.3);
\draw[thick] (4, 0) arc (0:-97:1.15 and 1.35);

\draw[thick] (0, 0) arc (180:360:3 and 3);

\draw[->, shift = {(4, 0)}] (2.5, 1.5) -- node[above] {$y$} (5, 1.5);
\draw[<->] (9.5, -2) -- (9.5, 5) node[right] {$\mathbb{R}$};
\draw[thick] (9.4, 1.5) -- (9.6, 1.5) node[right]{$0$};

\end{tikzpicture}
\caption{A $b^k$-manifold with disconnected $Z$}
\label{fig_disconnect}
\end{figure}
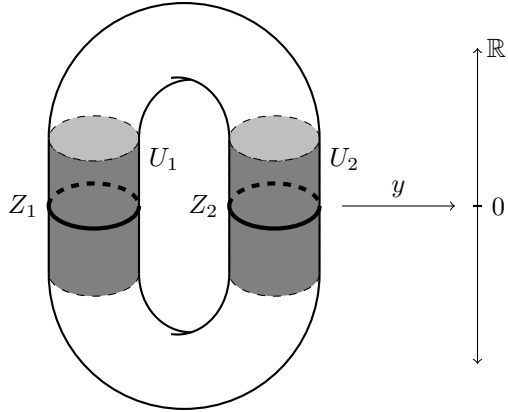

\noindent where a collar neighborhood $U = U_1 \cup U_2$ of $Z$ is shaded. Let $\{(\theta_i, y)\}$ be coordinates on $U_i$. Then $d\theta_1$ (respectively $d\theta_2$) extends trivially over $U_2$ (respectively $U_1$) to a form on all of $U$. Let $\omega$ be a $b^k$ 2-form on $M$. On $U$, it admits a decomposition
\[
\omega = \frac{dy}{y^k} \wedge (f d\theta_1 + gd\theta_2) + \beta
\]
for smooth functions $f, g$ and a smooth form $\beta$. Let $\pi: U \rightarrow Z$ be the vertical projection, and for $-k \leq i \leq -1$, let
\[
f_i := \frac{1}{(k+i)!} \restr{\frac{\partial^{k+i}f}{\partial y^{k+i}}}{Z}  \ \ \ \ \ g_i := \frac{1}{(k+i)!} \restr{\frac{\partial^{k+i}g}{\partial y^{k+i}}}{Z}.
\]
Then 
\begin{align*}
f &= \pi^*(f_{-k}) + \pi^*(f_{-k+1})y + \dots + \pi^*(f_{-1})y^{k-1} + \widetilde{f}\\
g &= \pi^*(g_{-k}) + \pi^*(g_{-k+1})y + \dots + \pi^*(g_{-1})y^{k-1} + \widetilde{g}
\end{align*}
for $\widetilde{f}, \widetilde{g} \in I^k$. Then $\omega$ has a Laurent series
\[
\omega = \sum_{i = 1}^k \frac{dy}{y^i} \wedge \left( \pi^*(f_i) d\theta_1 + \pi^*(g_i)d\theta_2 \right) + \beta'
\]
where $\beta'$ is smooth form.

\end{example}
\begin{proposition}\label{laurent_independence} The cohomology classes $[i^*(\alpha_{-i})] \in H^{p-1}(Z)$ appearing in a Laurent series of $\omega \in {^{b^k}}\Omega^p(M)$ depend only on $[\omega]$.
\end{proposition}
\begin{proof} By Proposition \ref{f1f2}, we may assume that all our Laurent series are written with respect to the same local defining function $y \in j_Z$. When $k = 1$, then for $\omega \in {^{b^1}}\hspace{-2pt}\Omega^{p}(M)$, the class $[i^*(\alpha_{-1})]$ is the image of $[\omega]$ in the map appearing in Equation \ref{chly_ses}, and therefore depends only on $[\omega]$.

For $k > 1$, assume the proposition is true for $k - 1$, and let $\omega \in {^{b^k}}\hspace{-2pt}\Omega^{p}(M)$. Consider Laurent series of two representatives of $[\omega]$,
\[
\omega_0 = \sum_{i = 1}^k \frac{dy}{y^i} \wedge \alpha_{-i} + \beta \ \ \ \ \textrm{and} \ \ \ \ \omega_1 = \sum_{i = 1}^k \frac{dy}{y^i} \wedge \alpha^{\prime}_{-i} + \beta^{\prime}
\]
Both $[i^*(\alpha_{-k})]$ and $[i^*(\alpha^{\prime}_{-k})]$ are the image of $[\omega]$ in Equation \ref{chly_ses}, so are equal. If we can show that
\[
\sum_{i = 1}^{k-1} \frac{dy}{y^i} \wedge \alpha_{-i} + \beta \ \ \ \ \textrm{and} \ \ \ \ \sum_{i = 1}^{k-1} \frac{dy}{y^i} \wedge \alpha^{\prime}_{-i} + \beta^{\prime}
\]
are cohomologous $b^{k-1}$-forms, then we will be done by induction. That is, we must show that
\begin{equation}\label{wanna_be_exact}
\omega_1 - \frac{dy}{y^k} \wedge \alpha^{\prime}_{-k} - \left( \omega_0 - \frac{dy}{y^k} \wedge \alpha_{-k}\right)
\end{equation}
is an exact $b^{k-1}$-form. Because $[\omega_0] = [\omega_1]$, there is a $b^{k}$-form $\eta$ with $d\eta = \omega_1 - \omega_0$. Moreover, because $\alpha_{-k} - \alpha^{\prime}_{-k}$ is a closed form with $i^*(\alpha_{-k} - \alpha^{\prime}_{-k})$ exact, the relative Poincar\'{e} lemma implies that it has a primitive $\mu$. Then
\[
\eta + \frac{dy}{y^k} \wedge \mu
\]
is a primitive for the form (\ref{wanna_be_exact}). However, this primitive is a $b^k$-form; to prove that (\ref{wanna_be_exact}) is exact as a $b^{k-1}$-form (and in doing so complete the induction), simply observe that the map 
\[
{^{b^{k-1}}}\hspace{-2pt}H^{p}(M) \rightarrow {^{b^{k}}}\hspace{-2pt}H^{p}(M)
\]
from Sequence \ref{chly_ses} is injective, so any $b^{k-1}$-form exact as a $b^k$-form is also exact as a $b^{k-1}$-form.
\end{proof}

\begin{corollary}
Let 
\[
\omega = \sum_{i = 1}^{k} \frac{dy}{y^i} \wedge {\alpha}_{-i} + \beta
\]
be a Laurent series of the closed $b^{k}$-form $\omega$. The map
\begin{align}\label{laurent_map}
{^{b^k}}\hspace{-2pt}H^p(M) &\rightarrow (H^{p-1}(Z))^k\\
[\omega] &\mapsto ([i^*(\alpha_{-1})], [i^*(\alpha_{-2})], \dots, [i^*(\alpha_{-k})])\notag
\end{align}
is independent of the choice of Laurent series.
\end{corollary}
\begin{definition}
Given a $b^k$-form $\omega$, the image of $[\omega]$ under Map \ref{laurent_map} is the {\bf Laurent Decomposition} of $[\omega]$.
\end{definition}
The result below strengthens Theorem \ref{noncanonicaliso}.
\begin{theorem} The sequence below, with $g, f$ given by the Map \ref{df_upgrade} and Map \ref{laurent_map} respectively, is exact.
\begin{equation}\label{laurent_ses}
0 \rightarrow H^p(M) \stackrel{g}{\rightarrow} {^{b^k}}\hspace{-2pt}H^p(M) \stackrel{f}{\rightarrow} (H^{p-1}(Z))^k \rightarrow 0
\end{equation}
\end{theorem}
\begin{proof} The map $g$ is the composition of the inclusions
\[
{^{b^{\ell-1}}}\hspace{-2pt}H^n(M) \rightarrow  {^{b^\ell}}\hspace{-2pt}H^n(M) 
\]
appearing in the short exact sequence (\ref{chly_ses}) for $\ell \leq k$. Therefore, it itself is an inclusion. The proof that $f$ is surjective follows from the same trick used to create a preimage of a closed $\alpha \in \Omega^{p-1}(Z)$ in the proof of Proposition \ref{ses_exact}. Exactness at the middle is straightforward.
\end{proof}

\section{Volume Forms on a $b^k$-manifold}\label{section_volume}

Let $(M, Z, j_Z)$ be a compact $b^k$-manifold, and let $\omega \in {^{b^k}}\hspace{-2pt}\Omega^{\textrm{dim}(M)}(M)$. Because $\omega$ ``blows up'' along $Z$, we cannot expect its integral to be finite. If we remove from $M$ a neighborhood of $Z$, then the integral of $\omega$ over the remainder is finite, but obviously depends on the choice of neighborhood. In this section, we extract a useful invariant of $\omega$ by studying the behavior of this integral as the size of the removed neighborhood shrinks. We will use this invariant to split the short exact sequence (\ref{laurent_ses}), and in doing so make the isomorphism (\ref{noncanonicaliso}) canonical.

The results from this section apply even to non-compact manifolds; so that we may state these results in full generality, we begin by introducing notation for compactly supported de Rham theory.
\begin{definition} The subset ${^{b^k}}\hspace{-2pt}\Omega^p_c(M) \subseteq {^{b^k}}\hspace{-2pt}\Omega^p(M)$ consists of $b^k$-forms with compact support. They form a subcomplex of the {$b^k$-de Rham complex}, the homology of which is called the {\bf compact $b^k$-cohomology} ${^{b^k}}\hspace{-2pt}H^*_c(M)$

\end{definition}

\subsection{Liouville Volume of a $b^k$-form}
\begin{definition} Let $(M, Z, j_Z)$ be an $n$-dimensional $b^k$-manifold. Given $\omega \in {^{b^k}}\hspace{-2pt}\Omega^n_c(M)$, $\epsilon > 0$ small, and a local defining function $y \in j_Z$, define $U_{y, \epsilon} = y^{-1}((-\epsilon, \epsilon))$ and 
\[
\textrm{vol}_{y, \epsilon}(\omega) = \int_{M \backslash U_{y, \epsilon}} \omega
\]
\end{definition}

In \cite{radko}, Radko proved that when $M$ is a surface\footnote{Although Radko studied $b$-forms only on surfaces, her proof of the fact stated here works for all $n$.} and $k = 1$, $\textrm{lim}_{\epsilon \rightarrow 0}\textrm{vol}_{y, \epsilon}(\omega)$ converges and is independent of $y$. This limit, the {\it Liouville volume} of $\omega$, was a key ingredient in her classification of stable Poisson structures on compact surfaces. When $k > 1$, this limit will not necessarily converge to a number, but rather to a polynomial in $\epsilon^{-1}$. After proving the existence and well-definedness of this polynomial, we will define the {\it Liouville volume} of a $b^k$-cohomology class of top degree as the constant term of this polynomial.

\begin{theorem}\label{thm_volumepolynomial} For a fixed $[\omega] \in {^{b^k}}H^n_c(M)$ on a $b^k$-manifold $(M, Z, j_Z)$ with $Z$ compact, there is a polynomial $P_{[\omega]}$ for which 
\begin{equation}\label{eqn_volumepolynomial}
\lim_{\epsilon \rightarrow 0} \left(P_{[\omega]} \left( \epsilon^{-1} \right) - \text{\textnormal{vol}}_{y, \epsilon}(\omega)\right) = 0
\end{equation}
for any $y \in j_Z$ and any $\omega$ representing $[\omega]$.
\end{theorem}

\begin{proof}
We first prove that there is a polynomial $P_{[\omega]}$ that satisfies Equation \ref{eqn_volumepolynomial} for a specific $y$ and $\omega$, then we prove that the polynomial is independent of $y$, then that the polynomial vanishes for exact $\omega$ (so depends only on the $b^k$-cohomology class of $\omega$).

Fix a local defining function $y \in  j_Z$ and a closed collar neighborhood $(y, \pi): U \rightarrow [-R, R] \times Z$ of $Z$. Because $\omega$ is compactly supported, $\int_{M \backslash U} \omega < \infty$, so to prove the existence of $P_{[\omega]}$ it suffices to construct a polynomial for the case $M = U$. By Remark \ref{proj_gives_ld}, there exists a Laurent series of $\omega$ of the form
\[
\omega = \sum_{i = 1}^{k} \frac{dy}{y^i} \wedge  \pi^*(\alpha_{-i}) + \beta
\]
where $\alpha_{-i} \in \Omega^{n-1}(Z)$. Then
\begin{align*}
\textrm{vol}_{y, \epsilon}(\omega) &= \int_{U \backslash U_{y, \epsilon}} \sum_{i = 1}^{k} \frac{dy}{y^i} \wedge  \pi^* (\alpha_{-i}) + \int_{U \backslash U_{y, \epsilon}} \beta
\end{align*}
Applying Fubini's theorem (and cancelling log terms), the first term simplifies to
\begin{align*}
\sum_{i = 2}^{k} \frac{-1}{i - 1}   &\left( \left(\frac{1}{R}\right)^{i - 1} + \left(\frac{-1}{\epsilon}\right)^{i - 1} - \left(\frac{-1}{R}\right)^{i - 1}  - \left(\frac{1}{\epsilon}\right)^{i - 1}  \right) \int_Z \alpha_{-i}\\ 
&=\sum_{\substack{i = 2 \\ i \ \textrm{even}}}^{k} \left( \frac{-2R^{1-i}}{i - 1}\right)\int_Z \alpha_{-i} +  \sum_{\substack{i = 2 \\ i \ \textrm{even}}}^{k} \left(\frac{2}{i - 1}\int_Z \alpha_{-i}\right) (\epsilon^{-1})^{i-1}
\end{align*}
and the last term simplifies to 
\[
\int_{U} \beta - \int_{[-\epsilon, \epsilon]\times Z} \beta
\]
so the polynomial 
\begin{align*}
P(t) &= \left(\int_{U}\beta + \sum_{\substack{i = 2 \\ i \ \textrm{even}}}^{k} \left( \frac{-2R^{1-i}}{i - 1}\right)\int_{Z}\alpha_{-i} \right) + \sum_{\substack{i = 2 \\ i \ \textrm{even}}}^{k} \left(\frac{2}{i-1}\int_Z \alpha_{-i}\right) t^{i-1}
\end{align*}
satisfies the conditions of a volume polynomial for this specific choice of $y$ and $\omega$.

The proof that this polynomial does not depend on $y$ is techincal; the details can be found in Section \ref{yucky_section}. To show that the polynomial associated to any exact form is trivial, suppose $\omega$ is exact and let 
\[
\eta = \sum_{i = 1}^{k} \frac{dy}{y^i} \wedge \pi^*\eta_{-i} + \beta_{\eta}
\]
be a Laurent series of a primitive of $\omega$. Then
\[
\int_{M \backslash U_{y, \epsilon}} \omega = \int_{\partial(M \backslash U_{y, \epsilon})} \eta = \int_{\partial (M \backslash U_{y, \epsilon})} \beta_{\eta}
\]
which approaches 0 as $\epsilon \rightarrow 0$.
\end{proof}

\begin{definition} The polynomial $P_{[\omega]}$ described in Theorem \ref{thm_volumepolynomial} is the {\bf volume polynomial} of $[\omega]$. Its constant term $P_{[\omega]}(0)$ is the {\bf Liouville volume} of $[\omega]$.
\end{definition}
The Liouville volume of $[\omega]$ can be thought of as the volume that remains of $[\omega]$ after its singular parts have been carefully discarded. For arbitrary (non-$b^k$) singularities of a form of top degree, no similar concept exists. In the $b^k$ case, the definition is made possible by how well-behaved $b^k$ singularities are, as well as by how we use the jet data (when $k > 1$) to prescribe the asymptotic manner in which $U_{y, \epsilon}$ approaches $Z$ as $\epsilon \rightarrow 0$.

We may also define the {\it Liouville volume} of a $p<\textrm{dim}(M)$ dimensional $b^k$-form $\omega$ along a compact $p$-dimensional submanifold $Y \subseteq M$ transverse to $Z$: the pullback of $\omega$ will be a $b^k$-form of top degree for the induced $b^k$-structure on $Y$ and therefore has a Louville volume. By Poincar\'{e} duality, this remark inspires the definition of the {\it smooth part} of a $b^k$-form.
\begin{definition}
Let $[\omega] \in {^{b^k}}\hspace{-2pt}H^{p}(M)$. The image of $[\omega]$ under the map
\begin{align}\label{splitting}
{^{b^k}}\hspace{-2pt}H^{p}(M) &\rightarrow (H^{n-p}_c(M))^* \cong H^p(M)\\
[\omega] &\mapsto \left( [\eta] \mapsto P_{[\omega \wedge \eta]}(0)\right) \notag
\end{align}
is its {\bf smooth part} $[\omega_{\textrm{sm}}] \in H^p(M)$.
\end{definition}
If $[\omega]$ is smooth (that is, $[\omega] \in \hspace{-2pt}H^{n}(M)  \subseteq {^{b^k}}\hspace{-2pt}H^{n}(M)$), then so too is $[\omega \wedge \eta]$ smooth for all $[\eta] \in (H^{n-p}_c(M))^*$. In this case, it follows that $P_{[\omega \wedge \eta]}(0)$ equals $\int_M \omega \wedge \eta$ and that $[\omega] = [\omega_{\textrm{sm}}]$. This remark shows that Equation \ref{splitting} splits the short exact sequence from Equation \ref{laurent_ses}, which yields a canonical isomorphism, the {\bf Liouville-Laurent} isomorphism, that realizes the (abstract) isomorphism from Proposition \ref{noncanonicaliso}.
\begin{align}\label{canonical_isomorphism}
\varphi: {^{b^k}}\hspace{-2pt}H^n(M) &\cong H^n(M) \oplus \left( H^{n-1}(Z) \right)^k\\
[\omega] &\mapsto ([\omega_{\textrm{sm}}], [\alpha_{-1}], \dots, [\alpha_{-k}])\notag
\end{align}
\begin{definition} Let $\omega$ be a $b^k$-form of top degree.  The {\bf Liouville-Laurent} decomposition of $[\omega]$ is its image under Equation \ref{canonical_isomorphism}, $([\omega_{\textrm{sm}}], [\alpha_{-1}], \dots, [\alpha_{-k}])$.
\end{definition}
The following proposition shows that taking the Liouville-Laurent decomposition of a $b^k$-form commutes with taking its pullback under a $b^k$-map.

\begin{proposition}\label{pullback_commute} Let $\varphi: (M, Z, j_Z) \rightarrow (M', Z', j_{Z'})$ be a $b^k$-map, and $[\omega'] \in {^{b^k}}\hspace{-2pt}H^p(M^{\prime})$ have Liouville-Laurent decomposition $([\omega_{\textrm{sm}}'], [\alpha_{-1}'], \dots, [\alpha_{-k}'])$. Then $[\varphi^*(\omega')]$ has Liouville-Laurent decomposition 
\[
([\varphi^*(\omega_{\textrm{sm}}')], [\restr{\varphi}{Z}^*(\alpha_{-1}')], \dots, [\restr{\varphi}{Z}^*(\alpha_{-k}')]).
\]
\end{proposition}
\begin{proof}
Let $y' \in j_{Z'}$, and $i_{Z}: Z \rightarrow M$, $i_{Z'}: Z' \rightarrow M'$ be the inclusions. By the definition of a $b^k$-map, $y := \varphi^*(y')$ represents $j_Z$. Then for a Laurent series of $\omega'$,
\[
\omega' = \sum_{i = 1}^{k} \frac{dy'}{y'^i} \wedge \pi^*\alpha_{-i}' + \beta',
\]
the pullback of $\omega'$ has Laurent series
\[
\varphi^*(\omega') = \sum_{i = 1}^{k} \frac{dy}{y^i} \wedge \varphi^*(\pi^*\alpha_{-i}') + \varphi^*(\beta').
\]
and we see that $[\varphi^*(\omega')]$ has Laurent decomposition 
\[
([i_{Z}^*(\varphi^*(\pi^*\alpha_{-1}'))], \dots, [i_{Z}^*(\varphi^*(\pi^*\alpha_{-k}'))]) = ([\restr{\varphi}{Z}^*(i_{Z'}^*(\pi^*\alpha_{-1}'))], \dots, [\restr{\varphi}{Z}^*(i_{Z'}^*(\pi^*\alpha_{-k}'))])
\]
which proves that the Laurent decomposition commutes with pullback.

Let $[\eta] \in H^{n-p}_c(M)$. To prove that $[\varphi^*(\omega')_{\textrm{sm}}] = [\varphi^*(\omega_{\textrm{sm}}')]$, it suffices to show that
\begin{equation}\label{eqn_pullback}
P_{[\varphi^*(\omega') \wedge \eta]}(0) = \int_M \varphi^*(\omega_{\textrm{sm}}') \wedge \eta.
\end{equation}

Our strategy for proving Equation \ref{eqn_pullback} will be to introduce an auxiliary family of smooth closed differential forms $\omega_{\epsilon}' \in \Omega^p(M')$ with the property that the Liouville volume of $\varphi^*(\omega') \wedge \eta$ can be calculated in terms of the asymptotic behavior of $\int_{M} \varphi^*(\omega_{\epsilon}') \wedge \eta$ instead of $\int_{M \backslash U_{y, \epsilon}} \varphi^*(\omega') \wedge \eta$. 

For $\epsilon > 0$ small, let $f_{\epsilon}: \mathbb{R} \rightarrow [0, 1]$ be a smooth function such that 
\[
\restr{f_{\epsilon}}{\mathbb{R} \backslash (-\epsilon, \epsilon)} = 1 \hspace{1cm} \textrm{and} \hspace{1cm} \restr{f_{\epsilon}}{(-\epsilon + \textrm{exp}(-\epsilon^{-1}), \epsilon - \textrm{exp}(-\epsilon^{-1}))} = 0
\]
and assume that $f_{\epsilon}$ varies smoothly with $\epsilon$. Define
\[
\omega_{\epsilon}' = \sum_{i = 1}^{k} f_{\epsilon}(y')\frac{dy'}{y'} \wedge \pi^* \alpha_{-i}' + \beta'
\]
and observe that $\omega_{\epsilon}'$ is closed and that $ \int_{M} \varphi^*(\omega_{\epsilon}') \wedge \eta $ approaches $\textnormal{vol}_{y, \epsilon}(\varphi^*(\omega') \wedge \eta)$ as $\epsilon \rightarrow 0$.

Next, recall that the pullback map in de Rham cohomology induces (by Poincar\'{e} duality) a pushforward map in compactly supported cohomology; we will use the notation $\varphi_*\eta$ for a representative of the pushforward of $[\eta] \in H^{n-p}_c(M)$. Using this notation,
\begin{align*}
0 &= \textrm{lim}_{\epsilon \rightarrow 0}\left(P_{[\varphi^*(\omega') \wedge \eta]}(\epsilon^{-1}) - \textnormal{vol}_{y, \epsilon}(\varphi^*(\omega') \wedge \eta)\right)\\
&= \textrm{lim}_{\epsilon \rightarrow 0}\left(P_{[\varphi^*(\omega') \wedge \eta]}(\epsilon^{-1}) - \int_{M} \varphi^*(\omega_{\epsilon}') \wedge \eta\right)\\
&= \textrm{lim}_{\epsilon \rightarrow 0}\left(P_{[\varphi^*(\omega') \wedge \eta]}(\epsilon^{-1}) - \int_{M'} \omega_{\epsilon}' \wedge \varphi_*\eta\right)\\
&= \textrm{lim}_{\epsilon \rightarrow 0}\left(P_{[\varphi^*(\omega') \wedge \eta]}(\epsilon^{-1}) - P_{[\omega' \wedge \varphi_*\eta]}(\epsilon^{-1})\right)\\
\end{align*}
so
\[
P_{[\varphi^*(\omega') \wedge \eta]}(0) = \int_{M'} \omega_{\textrm{sm}}' \wedge \varphi_*\eta = \int_M \varphi^*(\omega_{\textrm{sm}}') \wedge \eta
\]
which proves Equation \ref{eqn_pullback}.

%Both terms above are equal to the constant term of the polynomial in $\epsilon^{-1}$ given by the limiting behavior of 
%\[
%\int_{M' \backslash U_{y', \epsilon}} \omega \wedge \varphi_*\mu
%\]
%the first term because $\varphi(U_{y, \epsilon}) = U_{y', \epsilon}$ for small $\epsilon$, and the second because. 

%On each component of $M$, $\varphi(U_{y, \epsilon}) = U_{y', \epsilon}$, so if $\varphi$ is a diffeomorphism $\textrm{vol}_{y, \epsilon}(\varphi^*\omega) = \textrm{vol}_{y', \epsilon}(\omega)$ for small values of $\epsilon$. As a consequence, $[\varphi^*(\omega_{\textnormal{sm}})] = [(\varphi^*(\omega))_{\textnormal{sm}}]$.
\end{proof}

\subsection{$b^k$ Orientation}

The notion of orientability of a smooth manifold generalizes in an obvious way to the $b^k$-world.
\begin{definition} A {\bf volume $b^k$-form} on a $b^k$ manifold is a nowhere vanishing $b^k$-form of top degree. A $b^k$-manifold is {\bf $b^k$-orientable} if it admits a volume $b^k$-form. A {\bf $b^k$-orientation} on a connected orientable $b^k$-manifold is a choice of one of the two connected components of the space of volume $b^k$-forms. 
\end{definition}
Although the underlying smooth manifold of every $b^k$-manifold is orientable (an orientation for $M$ is included in the data of a $b^k$-manifold), not all $b^k$-manifolds are $b^k$-orientable. For example, if $Z \subseteq M$ is a meridian of the torus $S^1 \times S^1$ (so $M \backslash Z$ is connected), the corresponding $b^1$-manifold admits no volume $b^1$-form even though $M$ is orientable. The opposite is also true: if you remove from the definition of a $b^k$-manifold the condition that $M$ is oriented, then it remains possible to define the $b^k$-(co)tangent bundles, and according to these new definitions there would exist $b^k$-manifolds that admit a $b^k$-orientation even though the underlying manifold is unorientable. For example, if $Z \subseteq M$ is a meridian of the Klein bottle, there exists a volume $b^1$-form on the corresponding $b^1$-manifold even though $M$ is not orientable. Although it is possible to study the $b^k$-geometry of non-orientable manifolds by modifying the definition of a $b^k$-manifold in this way, omitting the data of an orientation makes it impossible to define the Liouville volume of a $b^k$-form of top degree. It is for this reason that we have restricted our attention to $b^k$-structures on oriented manifolds in this paper.

Notice that the image under $\iota_{\mathbb{L}}$ of a volume $b^k$-form $\omega$ will be a smooth volume form on $Z$. In this way, a $b^k$-orientation on $(M, Z, j_Z)$ induces an orientation on $Z$ which may or may not agree with the orientation of $Z$ given in the data of a $b^k$-manifold.
\begin{definition}\label{def_posoriented}
Let $\omega$ be a volume $b^k$-form on $(M, Z, j_Z)$. If the smooth form $\iota_{\mathbb{L}}(\omega)$ is positively oriented, we say that $\omega$ is a {\bf positively oriented} volume $b^k$-form.
\end{definition}

Notice that if $\omega$ is a volume $b^k$-form which is {\it not} positively oriented, one can replace the $b^k$ structure on $(M, Z, j_Z)$ with a different $b^k$ structure for which $\omega$ is a positively oriented volume $b^k$-form. To do so, reverse the orientations of those components $Z'$ of $Z$ for which $\restr{\iota_{\mathbb{L}}(\omega)}{Z'}$ is negatively oriented, and replace the jet data for those $Z'$ with their negatives.

%In Section \ref{section_symplectic}, we will introduce the notion of a symplectic $b^k$-form. Not surprisingly, a symplectic $b^k$-form will induce a volume $b^k$-form on $M$, and it will be convenient to assume that the orientations induced on $Z$ by the image of the 

%In the same way that the space of sections of nonvanishing smooth volume forms on a connected orientable manifold has two connected components (corresponding to the two possible orientations on the manifold), so too does the   on a manifold determines an orientation, so too is it possible to have an orie

\section{Symplectic and Poisson Geometry of $b^k$-Forms}\label{section_symplectic}

We begin this section by introducing the notion of a symplectic $b^k$-form and proving Moser's theorems in the $b^k$-category. We then classify symplectic $b^k$-surfaces, and show how the Liouville-Laurent decomposition of a $b$-symplectic form on a surface reconciles a classification theorem from \cite{gmp2} with one from \cite{radko}.
\begin{definition} A {\bf symplectic $b^k$-form} on a $b^k$-manifold is a closed $b^k$ 2-form having maximal rank at every $p \in M$.
\end{definition}
\begin{definition} A {\bf symplectic $b^k$-manifold} $(M, Z, j_Z, \omega)$ is a $b^k$-manifold $(M, Z, j_Z)$ with a symplectic $b^k$-form $\omega$. \end{definition}
\begin{definition} A {\bf $b^k$-symplectomorphism} $\varphi: (M, Z, j_Z, \omega) \rightarrow (M', Z', j_{Z'}, \omega')$ is a $b^k$-map satisfying $\varphi^*(\omega') = \omega$. \end{definition}
%\begin{remark}\label{rmk_surfaceorientation} When $\omega$ is a symplectic $b^k$-form on a surface, it is necessarily a volume $b^k$-form. As such, it the condition that the orientations on $M \backslash Z$ induced by $\omega$ and by $M$ agree implies that $\iota_{\mathbb{L}}(\omega)$ is a positively-oriented volume form on $Z$.\end{remark}
\begin{theorem} (relative Moser's theorem) If $\omega_0, \omega_1$ are symplectic $b^k$-forms on $(M, Z, j_Z)$ with $Z$ compact, $\restr{\omega_0}{Z} = \restr{\omega_1}{Z}$, and $[\omega_0] = [\omega_1]$, then there are neighborhoods $U_0, U_1$ of $Z$ and a $b^k$-symplectomorphism $\varphi: (U_0, Z, j_Z, \omega_0) \rightarrow (U_1, Z, j_Z, \omega_1)$ such that $\restr{\varphi}{Z} = \textnormal{id}$.
\end{theorem}
\begin{proof} Pick a local defining function $y \in j_Z$ and Laurent series of $\omega_{0}, \omega_{1}$
\[
\omega_0 = \sum_{i = 1}^{k} \frac{dy}{y^i} \wedge \alpha_{-i} + \beta
 \ \ \ \ \ \ \ \ \ \ \ \ 
\omega_1 = \sum_{i = 1}^{k} \frac{dy}{y^i} \wedge \alpha'_{-i} + \beta'.
\]
Then $i^*(\alpha'_{-i} - \alpha_{-i}) \in \Omega^1(Z)$ is exact for all $i$, and $i^*(\alpha'_{-k} - \alpha_{-k}) = i^*(\beta' - \beta) = 0$. By the relative Poincar\'{e} lemma there are primitives $\mu_i$ of $(\alpha'_{-i} - \alpha_{-i})$ and $\mu_{\beta}$ of $(\beta' - \beta)$ with $\restr{\mu_{-k}}{Z} = \restr{\mu_{\beta}}{Z} = 0$. Then $\omega_1 - \omega_0 = d\mu$, where 
\[
\mu = \sum_{i = 1}^{k} \frac{dy}{y^i} \wedge \mu_{-i}  + \mu_{\beta}\textrm{.}
\]
Let $\omega_t = t\omega_1 + (1-t)\omega_0$, and observe that $d{\omega_t}/dt = d\mu$. By shrinking our neighborhood around $Z$, we can assume that $\omega_t$ has full rank for all $t$, giving a pairing between $b^k$-vector fields and $b^k$ 1-forms. Because $\mu$ is a $b^k$ 1-form vanishing on $Z$ (since $\restr{\mu_{-k}}{Z} = 0$ and $\restr{\mu_{\beta}}{Z} = 0$), the vector field $v_t$ defined by Moser's equation
\[
\iota_{v_t} \omega_t = -\mu
\]
is a $b^k$-vector field that vanishes on $Z$, the time-one flow of which is the desired $b^k$-symplectomorphism. 
\end{proof}

\begin{theorem}\label{global_moser} (global Moser's theorem) Let $(M, Z, j_Z)$ be a compact $b^k$-manifold, and $\omega_t := t\omega_1 + (1-t)\omega_0$ a symplectic $b^k$-form for $t \in [0, 1]$, with $[\omega_0] = [\omega_1]$. Then there is an isotopy $\rho_t$ of $b^k$-maps with $\rho_t^*(\omega_t) = \omega_0$ for $t \in [0, 1]$. 
\end{theorem}
\begin{proof} Because $\frac{d\omega_t}{dt} = \omega_1 - \omega_0$ is exact, there is a smooth ${b^k}$-form $\mu$ such that $d\mu = \omega_1 - \omega_0$. Because $\omega_t$ is a $b^k$-form, it defines an pairing between $b^k$ 1-forms and $b^k$-vector fields. Therefore, the vector field $v_t$ defined by Moser's equation
\[
\iota_{v_t}\omega_t = -\mu
\]
is a $b^k$-vector field, so its flow defines an isotopy $\rho_t$ of $b^k$-maps with $\rho_t^*(\omega_t) = \omega_0$.
\end{proof}

\subsection{Classification of Symplectic $b^k$-Surfaces}

In \cite{radko}, the author classifies the space of stable Poisson structures on a connected, compact surface in terms of geometric data. In \cite{gmp2}, the authors demonstrate a correspondence between stable Poisson structures and $b$-symplectic forms on a manifold, and classify $b$-symplectic forms on a connected, compact surface in terms of their $b$-cohomology class. Pictorially, we have two sides of the triangle

\begin{center}
\begin{tikzpicture}
\draw (5, 0) node (box1){
$\left\{ \begin{array}{c} \textrm{Symplectic} \\ b\textrm{-forms on} \ (M, Z)  \end{array} \middle\} \right/ b\textrm{-symp.}$
};

\draw (2, 2.5) node (box2){
$\left\{ \begin{array}{c} \textrm{L. Vol} \in \mathbb{R} \\ \{\textrm{pd}(\gamma_i)\}_{i = 1}^r \in \mathbb{R}_{>0}^r \end{array} \right\}$
};

\draw (8, 2.5) node (box3){
$^{b}H^2(M)$
};

\draw[<->] (3, 2) -- node[above, rotate=-45] {\cite{radko}} (4.5, 0.5);
\draw[<->] (7, 2) -- node[above, rotate=45] {\cite{gmp2}} (5.5, 0.5);
\end{tikzpicture} \end{center}

\noindent where $M$ is a connected, compact surface, $\{\gamma_i\}$ are the $r$ oriented circles that constitute $Z$, $\textrm{L. Vol} \in \mathbb{R}$ is the Liouville volume of $(M, Z, \omega)$, and $\textrm{pd}(\gamma_i)$ is the period of the modular vector field on $\gamma_i$.

Theorem \ref{aoeuaoeu} completes the triangle. That is, it exhibits a direct connection between the cohomological classification data in \cite{gmp2} and the geometric classification data in \cite{radko}.

\begin{theorem}\label{aoeuaoeu} Let $[\omega] = ([\omega_{\textrm{sm}}], [\alpha_{-1}])$ be the Liouville-Laurent decomposition of a positively oriented $b$-symplectic form on a connected compact surface. Let $\{\gamma_r\}$ be the oriented circles that constitute $Z$. Then the Liouville volume of $\omega$ is $\int_{M}\omega_{\textrm{sm}}$, and the period of the modular vector field on $\gamma_r$ is
\[
\left(\int_{\gamma_r} \alpha_{-1}\right)^{-1}
\]
\end{theorem}
\begin{proof}
The fact that the Liouville volume of $\omega$ equals $\int_{M}\omega_{\textrm{sm}}$ follows from the definition of the smooth part of a $b^k$-form. Let $\gamma_i$ be a connected component of $Z$. We can find a collar neighborhood
\begin{align*}
U = \{(y, \theta), |y| < R, \theta \in [0, 1]/\sim \} \ \ \ \ \ \ \ \ R > 0
\end{align*}
such that on $U$
\[
\omega = c\frac{dy}{y} \wedge d\theta \ \ \ \ \ \ \ c > 0
\]
where $d\theta$ is a positively-oriented volume form on $Z$. From \cite{radko}, we know that the period of the modular vector field is $c^{-1}$, and we calculate that 
\[
\int_{\gamma_i} \alpha_{-1} = \int_{\gamma_i} cd\theta = c\textrm{.}
\]
\end{proof}

\begin{theorem}\label{bk_triangle} Let $\omega_0, \omega_1$ be symplectic $b^k$-forms on a compact connected $b^k$-surface $(M, Z, j_Z)$. The following are equivalent
\begin{enumerate}
\item There is a $b^k$-symplectomorphism $\varphi: (M, Z, j_Z, \omega_0) \rightarrow (M, Z, j_Z, \omega_1)$.
\item $[\omega_0] = [\omega_1]$
\item The Liouville volumes of $\omega_0$ and $\omega_1$ agree, as do the numbers 
\[
\int_{\gamma_r} \alpha_{-i}
\]
for all connected components $\gamma_r \subseteq Z$ and all $1 \leq i \leq k$, where $\alpha_{-i}$ are the terms appearing in the Laurent decomposition of the two forms.
\end{enumerate}
\end{theorem}
\begin{proof} \hspace{1pt}

\begin{description}
\item[$(1) \iff (2)$] This follows from the global Moser's Theorem (Theorem \ref{global_moser}) in dimension 2.
\item[$(2) \iff (3)$] The isomorphism (\ref{canonical_isomorphism}) shows that the cohomology class of a volume $b^k$ form is determined by its Liouville-Laurent decomposition, which in turn is determined by its Liouville volume and the integrals $\int_{\gamma_r} \alpha_{-i}$.
\end{description}
\end{proof}

\subsection{Integrable systems and a $b^k$ Poincar\'{e} Lemma}
We began Section \ref{sec_derham} by studying the following complex of sheaves.
\begin{equation}\label{complex_bkderham}
0 \rightarrow C^{\infty} \rightarrow {^{b^k}}\hspace{-2pt}\Omega^1 \rightarrow {^{b^k}}\hspace{-2pt}\Omega^2 \rightarrow {^{b^k}}\hspace{-2pt}\Omega^3 \rightarrow \dots
\end{equation}
Although this complex acts as a generalization of the de Rham complex of smooth forms, there is one remarkable dissimilarity between them: unlike in the smooth case, there exist $b^k$-forms that are closed but not locally exact. This is true even when $k = 1$. For example, if $y$ is a local defining function for $Z$, then $y^{-1}dy$ is a closed $b$-form, but in no neighborhood of any $p\in Z$ is it exact. Informally, it {\it wants} to be the differential of $\textrm{log}(y)$, but $\textrm{log}(y)$ is not a section of $C^{\infty}$ in any neighborhood of $p$. This ``failure of the Poincar\'{e} lemma'' plagues only $b^k$ $1$-forms, not $b^k$ forms of higher degree. Indeed, if 
\[
\omega = \sum_{i = 1}^k \frac{dy}{y^i} \wedge \pi^*(\alpha_i) + \beta
\]
is a Laurent decomposition of a $b^k$ form $\omega$ of degree $\geq 2$ in a collar neighborhood of $Z$, then for a sufficiently small neighborhood $U$ of any $p \in Z$, there exist primitives $\eta_i$ of $\restr{\alpha_i}{U \cap Z}$ and $\eta_{\beta}$ of $\restr{\beta}{U}$, and
\[
\sum_{i = 1}^k \frac{dy}{y^i} \wedge \pi^*(\eta_i) + \eta_{\beta}
\]
is a primitive of $\omega$. In light of this fact, one way to make peace with this failure of Poincar\'{e}'s lemma is to enlarge the sheaf $C^{\infty}$ into a new sheaf ${^{b^k}}\hspace{-2pt}C^{\infty}$ whose sections, in any contractible neighborhood, include primitives of all closed $b^k$ 1-forms. 
\begin{definition}
Let $(M, Z, j_Z)$ be a $b^k$-manifold. The sheaf ${^{b^k}} \hspace{-2pt}C^{\infty}$ on $M$ is defined by
\[
{^{b^k}}\hspace{-2pt}C^{\infty}(U) := \{ f \in C^{\infty}(U \backslash Z) \mid df \ \textrm{extends to an element of} \ {^{b^k}}\hspace{-2pt}\Omega^1(U)\}
\]
and the differential ${^{b^k}}\hspace{-2pt}C^{\infty} \rightarrow {^{b^k}}\hspace{-2pt}\Omega^1$ is the map that sends $f \in {^{b^k}}\hspace{-2pt}C^{\infty}(U)$ to the unique element of ${^{b^k}}\hspace{-2pt}\Omega^1(U)$ that extends $df$.
\end{definition}
After replacing $C^{\infty}$ with ${^{b^k}}\hspace{-2pt}C^{\infty}$ in the complex \ref{complex_bkderham}, every closed $b^k$ 1-form is locally exact. The question of whether $C^{\infty}(M)$ or ${^{b^k}}\hspace{-2pt}C^{\infty}(M)$ is the correct notion of ``function'' on a $b^k$-manifold is more than a superficial convention -- it raises the question of how to define an integrable system on a $b^k$-manifold. If $\omega$ is a symplectic $b^k$-form and $f \in C^{\infty}(M)$, then the symplectic gradient $X_f$ of $f$ will be a vector field whose restriction to $Z$ is tangent to the leaves of the symplectic foliation of $Z$ defined by $\textrm{ker}(\iota_{\mathbb{L}}\omega)$. On the other hand, if $f \in {^{b^k}}\hspace{-2pt}C^{\infty}(M)$, then $X_f$ will still be tangent to $Z$, but not necessarily tangent to the leaves of the symplectic foliation. This difference can also be seen in the fact that for any $p$, the map 
\begin{align}\label{eqn_surje}
{^{b^k}}\hspace{-2pt}C^{\infty}(M) &\rightarrow {^{b^k}}\hspace{-2pt}T_p(M)\\
f &\mapsto (X_f)_p\notag
\end{align}
is surjective (generalizing the surjection $C^{\infty}(M) \rightarrow T_p(M)$ from the classic theory of integrable systems), but this surjectivity would fail if ${^{b^k}}\hspace{-2pt}C^{\infty}(M)$ were replaced with $C^{\infty}(M)$. So far, the only definitions of an ``integrable system on a $b$-manifold'' that the author of this paper is aware of have taken $C^{\infty}$ functions to be the integrals of motion. It would be exciting to also study integrable systems whose integrals of motion are ${^{b^k}}\hspace{-2pt}C^{\infty}$ functions.

\section{Symplectic and Poisson structures of $b^k$-type}

When the authors of \cite{gmp2} studied the Poisson structures dual to symplectic $b$-forms, they found that $b$-symplectomorphisms are precisely Poisson isomorphisms of the dual Poisson manifolds. Unfortunately, this observation does not generalize to the $b^k$ case: although every symplectic $b^k$-form is dual to a Poisson bivector, not every Poisson isomorphism (with respect to this bivector) is realized by a $b^k$-map. Similarly, if $(M, Z, j_Z, \omega)$ and $(M, Z, j_Z', \omega')$ are two symplectic $b^k$-manifolds, there may be a diffeomorphism of $(M, Z)$ that restricts to a symplectomorphism $(M \backslash Z, \omega) \rightarrow (M \backslash Z, \omega')$ even if there is no $b^k$-symplectomorphism $(M, Z, j_Z, \omega) \rightarrow (M, Z, j_Z', \omega')$. In this section, we show how to use $b^k$-manifolds to prove statements about objects outside of the $b^k$-category. We begin by defining the notion of a Poisson (and symplectic) structure of $b^k$-type -- these are the Poisson (and symplectic) structures that are dual to (or equal to) a symplectic $b^k$-form for {\it some} choice of jet data. Then we apply the theory of symplectic $b^k$-forms to classify these structures on compact connected surfaces.

\begin{definition} Let $Z$ be an oriented hypersurface of an oriented manifold $M$. Let $\Pi$ be a Poisson structure on $M$ having full rank on $M \backslash Z$, and let $\omega \in \Omega^2(M \backslash Z)$ be the symplectic form dual to $\restr{\Pi}{M \backslash Z}$. We say that $\Pi$ and $\omega$ are {\bf of $b^k$ type} if there is some $j_Z \in J^{k-1}$ for which $(M, Z, j_Z)$ is a $b^k$-manifold on which $\omega$ extends to a symplectic $b^k$-form.
\end{definition}

\begin{remark}
Notice that if $\Pi$ is a Poisson structure of $b^k$-type on $(M^{2n}, Z)$ with dual form $\omega$, then there will be several distinct jets with respect to which $\omega$ is a symplectic $b^k$-form. For example, if $j_Z = [y]$ is one such jet and $f: \mathbb{R} \rightarrow \mathbb{R}$ satisfies $f(0) = 0$ and $f'(0) > 0$, then the jet $j_Z' := [f \circ y]$ defines exactly the same $b^k$-(co)tangent bundles as $j_Z$. As such, $\omega$ is a symplectic form with respect to both $j_Z'$ and $j_Z$. However, one can check that the condition of $\omega^n$ being positively oriented (as a volume $b^k$-form in the sense of Definition \ref{def_posoriented}) does {\it not} depend upon the chosen jet. Therefore, we say that $\Pi$ (or $\omega$) is a {\bf positively oriented} Poisson structure (or symplectic form) of $b^k$ type if $\omega$ extends to a positively oriented volume $b^k$-form for {\it any} choice of jet $j_Z$ for which $\omega$ extends to a $b^k$ form.
\end{remark}

To study Poisson and symplectic structures of $b^k$-type using the tools of $b^k$-geometry, we must understand how a $b^k$-form behaves under diffeomorphisms of $(M, Z)$ that are not necessarily $b^k$-maps. Of particular interest to us will be diffeomorphisms of $M$ that restrict to $(z, y) \mapsto (z, P(y))$ in a collar neighborhood $Z \times \mathbb{R}$ of $Z$, where $P$ is a polynomial. The following proposition describes how the Liouville-Laurent decomposition behaves under pullback of such a map (compare this proposition to Proposition \ref{pullback_commute}, where we showed that the Liouville-Laurent decomposition commutes with the pullback of a $b^k$-map).

\begin{proposition}\label{poly_change}
Let $P$ be a polynomial with $P(0) = 0$ and $P'(0) > 0$. Let $(M, Z, j_Z)$ be a $b^k$-manifold with positively oriented local defining function $y \in j_Z$, and let $\varphi: M \rightarrow M$ be a diffeomorphism given by $\textnormal{id} \times P(y)$ in a collar neighborhood $(\pi, y): U \rightarrow Z \times \mathbb{R}$ of $Z$. Then
\begin{itemize}
\item If $\omega$ is a $b^k$-form, then $\varphi^{*}(\omega)$ is also a $b^k$-form on $(M, Z, j_Z)$.
\item If $[\omega]$ has Liouville-Laurent decomposition $([\omega_{\textrm{sm}}], [\alpha_{-1}], \dots, [\alpha_{-k}])$ and $[\varphi^*(\omega)]$ has Laurent decomposition $([\omega_{\textrm{sm}}'], [\alpha_{-1}'], \dots, [\alpha_{-k}'])$, then $[\varphi^*(\omega_{\textrm{sm}})] = [\omega_{\textrm{sm}}']$ and $[\alpha_{-1}] = [\alpha_{-1}']$. 
\end{itemize}
\end{proposition}
\begin{proof}
In a collar neighborhood, let
\[ \omega = \sum_{i = 1}^{k} \frac{dy}{y^i} \wedge \pi^{*}(\alpha_{-i}) + \beta \]
be a Laurent decomposition of $\omega$. Then
\begin{equation}\label{expand_laurent} \varphi^{*}(\omega) = \sum_{i = 1}^{k} \frac{P'(y)dy}{P(y)^i} \wedge \pi^{*}(\alpha_{-i}) + \varphi^*\beta. \end{equation}
Notice that each term $\frac{P'(y)}{P(y)^i}$ must have a Laurent series with no exponents less than $-i$: indeed, 
\[
y^i\frac{P'(y)}{P(y)^i} = \left(\frac{y}{P(y)}\right)^iP'(y)
\]
is smooth. By replacing each $\frac{P'(y)}{P(y)^i}$ in equation \ref{expand_laurent} with its Laurent series, this proves the first claim. To prove the second claim, first observe that for $i \neq 1$, 
\[
\frac{P'(y)dy}{P(y)^i} = d\left( \frac{1}{-i + 1} P(y)^{-i+1} \right)
\]
so the meromorphic function $P'(y)P(y)^{-i}$ has no residue. For $i = 1$ the function $P'(y)P(y)^{-1}$ has a Laurent series with principal part $1/y$. Therefore, by replacing the $P'(y)P(y)^{-i}$
terms in Equation \ref{expand_laurent} with their Laurent series in the variable $y$, we arrive at a Laurent series of $\varphi^*(\omega)$ that has $y^{-1}dy\wedge \pi^*(\alpha_{-1})$ as its residue term, proving that $[\alpha_{-1}] = [\alpha_{-1}']$. To prove that $[\varphi^*(\omega_{\textrm{sm}})] = [\omega_{\textrm{sm}}']$, let $[\eta] \in {^{b^k}}\hspace{-2pt}H^{n-p}_c(M)$, where $p$ is the degree of $\omega$ and $n = \textrm{dim}(M)$. It suffices to show that 
\begin{equation}\label{eqn_toprove}
P_{[\varphi^*(\omega) \wedge \eta]}(0) = \int_{M} \varphi^*(\omega_{\textrm{sm}}) \wedge \eta \textrm{.}
\end{equation}
Towards this goal, observe that for $\epsilon > 0$ small, $\varphi(U_{y, \epsilon}) = U_{y, P(\epsilon)}$, so $\textnormal{vol}_{y, \epsilon}(\varphi^*(\omega \wedge (\varphi^{-1})^{*}\eta)) = \textnormal{vol}_{y, P(\epsilon)}(\omega \wedge (\varphi^{-1})^{*}\eta))$. Then letting 
\[
\omega \wedge (\varphi^{-1})^*\eta = \sum_{i = 1}^k \frac{dy}{y} \wedge \pi^*(\widetilde{\alpha}_{-i}) + \widetilde{\beta}
\]
be a Laurent series of $\omega \wedge (\varphi^{-1})^*\eta$, 
\begin{align*}
\textnormal{vol}_{y, \epsilon}(\varphi^*(\omega) \wedge \eta) - \textnormal{vol}_{y, \epsilon}(\omega \wedge (\varphi^{-1})^*\eta) 
&= \left( \int_{M \backslash U_{y, P(\epsilon)}} - \int_{M \backslash U_{y, \epsilon}} \right) \omega \wedge (\varphi^{-1})^*\eta\\
&= \int_Z \left( \int_{P(\epsilon)}^\epsilon - \int_{P(-\epsilon)}^{-\epsilon} \right)\sum_{i = 1}^k \frac{dy}{y^i}\pi^*(\widetilde{\alpha}_{-i}) \\
& \ \ \ \ \ \ \ \ \  +  \left( \int_{M \backslash U_{y, P(\epsilon)}} - \int_{M \backslash U_{y, \epsilon}} \right) \widetilde{\beta}\\
\end{align*}
As $\epsilon \rightarrow 0$, this limit approaches an odd function of $\epsilon$, proving that $P_{[\varphi^*(\omega) \wedge \eta]}(0) = P_{[\omega \wedge (\varphi^{-1})^*\eta]}(0)$, from which Equation \ref{eqn_toprove} follows.
\end{proof}

\begin{lemma}\label{poly_pick}
Let $(a_{-1}, \dots, a_{-k}) \in \mathbb{R}^k$ with $a_{-k} > 0$. There is a polynomial $P = \sum p_iy^i$ with $p_0 = 0$ and $p_1 > 0$ satisfying
\[
\sum_{i = 1}^{k} a_{-i} \frac{P'(y)}{P(y)^{i}} = \frac{1}{y^k} + \frac{a_{-1}}{y} + Q(y)
\]
where $Q(y)$ is a polynomial.
\end{lemma}
\begin{proof}
The proof is technical. See Section \ref{yucky_section} for the details.
\end{proof}
The two results above are the ingredients we need to prove the main theorem of this section.
\begin{theorem} Let $Z$ be an oriented hypersurface of a compact oriented surface $M$. Let $\Pi, \Pi'$ be two positively oriented Poisson structures of $b^k$-type on $(M, Z)$, and $\omega, \omega'$ be the dual $b^k$-symplectic forms (with respect to possibly different $b^k$-structures) with Liouville-Laurent decompositions
\begin{align*}
[\omega] &= ([\omega_{\textrm{sm}}], [\alpha_{-1}], \dots, [\alpha_{-k}])\\
[\omega'] &= ([\omega^{\prime}_{\textrm{sm}}], [\alpha^{\prime}_{-1}], \dots, [\alpha^{\prime}_{-k}]).
\end{align*}
If $[\omega^{\prime}_{\textrm{sm}}] = [\omega_{\textrm{sm}}] \in H^2(M)$ and $[\alpha^{\prime}_{-1}] = [\alpha_{-1}] \in H^1(Z)$, then there is a Poisson isomorphism  $\varphi: (M, \Pi) \rightarrow (M, \Pi')$.
\end{theorem}

\begin{proof}
Let $j_Z$ and $j'_Z$ be the jets of $Z$ with respect to which $\omega$ and $\omega'$ respectively are $b^k$-forms with the described Liouville-Laurent decompositions, and let $y \in j_Z, y' \in j'_Z$ be positively oriented local defining functions for $Z$. Let $\{\gamma_{\ell}\}$ be the oriented circles that constitute the connected components of $Z$. If 
\begin{align*}
&\varphi: U_{\ell} \rightarrow \mathbb{R} \times \mathbb{S}^1 = \{(y, \theta)\}\\
&\varphi': U_{\ell} \rightarrow \mathbb{R} \times \mathbb{S}^1 = \{(y', \theta)\}
\end{align*}
are local coordinate charts for a collar neighborhood $U_{\ell}$ of $\gamma_{\ell}$, then the map $(\varphi')^{-1} \circ \varphi$ is an orientation-preserving map in a neighborhood of $\gamma_i$, restricts to the identity on $\gamma_i$, and pulls $j_{Z}'$ back to $j_Z$. As such, the collection of these maps (one for each $\gamma_{\ell} \subseteq Z$) defines a smooth map in a neighborhood of $Z$ that extends to a $b^k$-diffeomorphism $(M, Z, j_Z) \rightarrow (M, Z, j'_Z)$. By replacing $\omega'$ with its pullback under this $b^k$-diffeomorphism and citing Proposition \ref{pullback_commute}, we may assume that $\omega, \omega'$ are $b^k$-symplectic forms on the {\it same} $b^k$-manifold $(M, Z, j_Z)$, and that the Liouville-Laurent decomposititons of $\omega, \omega'$ with respect to this $b^k$ structure are as described in the theorem statement.

Let $\pi: U_{\ell} = \{(y, \theta_{\ell})\} \rightarrow S^1$ be projetion onto the second coordinate. We may assume (by the global Moser's theorem) that
\[
\restr{\omega}{U_{\ell}} = \sum_{i = 1}^{k} \frac{dy}{y^i} \wedge a_i\pi^*(d\theta_{\ell}) + \beta_0
\]
where $a_i \in \mathbb{R}$ and $a_{-k} > 0$ (because $\Pi, \Pi'$ are positively oriented). Then we may apply Lemma \ref{poly_pick} to choose a polynomial $P_{\ell} = \sum p_iy^i$ with $p_0 = 0, p_1 > 0$ satisfying
\[
\sum_{i = 1}^{k} a_{-i} \frac{P'(y)}{P(y)^{i}} = \frac{1}{y^k} + \frac{a_{-1}}{y} + Q_{\ell}(y)
\]
for some polynomial $Q_{\ell}(y)$. By replacing $\omega$ with its pullback under a diffeomorphism of $(M, Z)$ that is of the form $(y, \theta_{\ell}) \mapsto (P_{\ell}(y), \theta_{\ell})$ in each $U_{\ell}$, we may assume $[\omega]$ has Liouville-Laurent decomposition 
\[
([\omega_{\textrm{sm}}], [\alpha_{-1}], 0, \dots, 0, [d\theta])
\]
where $d\theta$ is the form on $Z$ that restricts to $d\theta_i$ on each $\gamma_i$. Similarly, we may replace $\omega'$ with a form also having this Liouville-Laurent decomposition. Finally, we apply the global Moser's theorem (Theorem \ref{global_moser}) and the fact that $M$ is a surface to complete the proof.
\end{proof}

\section{Proof of Technical Results}\label{yucky_section}

\subsection{Proof of Proposition \ref{f1f2}}

\begin{proof} The case $k = 1$ was covered in \cite{gmp2}, so we may assume $k \geq 2$. Because $[f_1]^{k-1} = [f_2]^{k-1}$, we have $f_1 = f_2(1 + gf_2^{k-1})$ for a smooth function $g$. Note that $(1+gf_2^{k-1})^{-1} = (1 + g'f_2^{k-1})$ for $g' = -g(1 + gf_2^{k-1})^{-1}$. Then
\begin{align*} 
\frac{df_1}{f_1^k} &= \frac{df_2}{f_2^k(1+gf_2^{k-1})^{k-1}} + \frac{d(1 + gf_2^{k-1})}{f_2^{k-1}(1 + gf_2^{k-1})^k}\\
&= (1 + g'f_2^{k-1})^{k-1}\frac{df_2}{f_2^k} +  \frac{(k-1)gdf_2}{(1 + gf_2^{k-1})^kf_2}  + \beta' \\
&= (1 + (k-1)g'f_2^{k-1})\frac{df_2}{f_2^k} +  \frac{(k-1)gdf_2}{(1 + gf_2^{k-1})^kf_2}  + \beta'' \\
&= \frac{df_2}{f_2^k} + (k-1)g(-(1 + gf_2^{k-1})^{-1} + (1 + g'f_2^{k-1})^k)\frac{df_2}{f_2} +  \beta'' \\
&= \frac{df_2}{f_2^k} + (k-1)g(-(1 + g'f_2^{k-1}) + (1 + g'f_2^{k-1})^k)\frac{df_2}{f_2} +  \beta'' \\
&= \frac{df_2}{f_2^k} + \beta \\
\end{align*}
where 
\begin{align*}
\beta' &= (1 + g'f_2^{k-1})^{-k}dg\\
\beta'' &= \beta' + \sum_{i = 2}^{k-1} \binom{k-1}{i}(g'f_2^{k-1})^i \frac{df_2}{f_2^k}\\
\beta &= \beta'' + (k-1)g(-g'f_2^{k-1} + \sum_{i = 1}^k \binom{k}{i} (g'f_{2}^{k-1})^i)\frac{df_2}{f_2}
\end{align*}
\end{proof}

\subsection{Proof of subclaim of Theorem \ref{thm_volumepolynomial}}

We begin the proof with two technical lemmas.
\begin{lemma}
If $f(x): \mathbb{R} \rightarrow \mathbb{R}$ satisfies $[f]_0^{k-1} = [x]_0^{k-1}$, then its inverse $h: (-\epsilon, \epsilon) \rightarrow \mathbb{R}$ satisfies $[h]_0^{k-1} = [x]_0^{k-1}$.
\end{lemma}
\begin{proof} Because $[f]^{k-1} = [x]^{k-1}$, $f = (x + g(x)x^{k})$ for some smooth $g$. Then
\[
x = f(h(x)) = h(x) + g(h(x))h(x)^k\textrm{.}
\]
Because $h(x)$ vanishes at 0 (since $f$ does), $x - h(x)$ vanishes to order at least $k$, so $[h]^{k-1} = [x]^{k-1}$.
\end{proof}
\begin{lemma} Let $f: \mathbb{R} \rightarrow \mathbb{R}$ satisfy $[f]_0^{k-1} = [x]_0^{k-1}$. Then for all $i \leq k-1$
\begin{equation}\label{4term}
\frac{1}{x^i} - \frac{1}{(-x)^i} + \frac{1}{f(-x)^i} - \frac{1}{f(x)^i}
\end{equation}
is a smooth function that vanishes at $0$.
\end{lemma}
\begin{proof} Because $[f]^{k-1} = [x]^{k-1}$, $f(x) = x(1+gx^{k-1})$ for some smooth $g$. Then 
\[
h(x) := \frac{1}{x^i} - \frac{1}{f(x)^i} = \frac{(1 + gx^{k-1})^i - 1}{x^i(1 + gx^{k-1})^i} = \frac{(\sum_{j = 1}^{i}  \binom{i}{j} g^jx^{j(k - 1) - i} ) }{(1 + gx^{k-1})^i}
\]
is a smooth function. Equation \ref{4term} equals $h(x) - h(-x)$, so it is a smooth odd function, hence vanishes at zero.
\end{proof}

\begin{proof} (of subclaim of Theorem \ref{thm_volumepolynomial})
Let $U$ be a tubular neighborhood $(y, \pi): U \rightarrow [-R, R] \times Z$, with $y \in j_Z$. Let $h$ be another element of $j_Z$. It suffices to show that 
\[
\lim_{\epsilon \rightarrow 0} \left( \textnormal{vol}_{h, \epsilon}(\omega) - \textnormal{vol}_{y, \epsilon}(\omega)\right) = 0
\]
for the case $M = U$. To do so, let $y_{h, z}: \mathbb{R} \rightarrow \mathbb{R}$ be the function, defined near zero, inverse to $\restr{h}{ [-R, R] \times \{z\}}$. That is, for sufficiently small $\epsilon$, $h(y_{h, z}(\epsilon), z) = \epsilon$ and
\[
U_{h, \epsilon} = \{(y, z) \in [-R, R] \times Z \mid y_{h, z}(-\epsilon) \leq y \leq y_{h, z}(\epsilon)\}
\]
Then
\begin{align*}
 \textrm{vol}_{h, \epsilon}(\omega) &- \textrm{vol}_{y, \epsilon}(\omega) = \left(\int_{U \backslash U_{h, \epsilon}}  - \int_{U \backslash U_{y, \epsilon}} \right) \omega\\
&= \int_{Z}\left(\int_{y_{h, z}(\epsilon)}^R + \int_{-R}^{y_{h, z}(-\epsilon)} -\int_{\epsilon}^{R} - \int_{-R}^{-\epsilon}\right) \sum_{i = 1}^{k}\frac{dy}{y^i} \pi^{*}(\alpha_{-i})\\
& \ \ \ + \left(\int_{U \backslash U_{h, \epsilon}}  - \int_{U \backslash U_{y, \epsilon}} \right) \beta \\
&= \int_{Z}\left(\log\left|\frac{y_{h, z}(-\epsilon)}{y_{h, z}(\epsilon)}\right| + \log\left|\frac{\epsilon}{-\epsilon}\right| \right)\pi^*(\alpha_{-1}) + \\
& \ \ \ + \sum_{i = 2}^{k} \frac{1}{1-i} \int_Z \left( -(y_{h, z}(\epsilon))^{1-i} + (y_{h, z}(-\epsilon))^{1-i} + (\epsilon)^{1-i} - (-\epsilon)^{1-i} \right)\pi^*(\alpha_{-i})\\
& \ \ \ + \int_{U_{y, \epsilon}} \beta - \int_{U_{u, \epsilon}}\beta
\end{align*}
by the previous lemmas, the limit as $\epsilon \rightarrow 0$ of the above expression is zero, proving the claim.
\end{proof}
\subsection{Proof of Lemma \ref{poly_pick}}
\begin{proof}
Recall from the proof of Proposition \ref{poly_change} that for any polynomial $P$ and $i \neq 1$, the expression
\[
\frac{P'(y)}{P(y)^{i}}
\]
has a Laurent series in $y$ with trivial residue term and no exponents less than $-i$. When $i = 1$, the same expression has principal part $y^{-1}$. Therefore, for any polynomial $P$,
\begin{equation}\label{eqn_bees}
\sum_{i = 1}^{k} a_{-i} \frac{P'(y)}{P(y)^{i}} = \sum_{i = 2}^{k} \frac{b_{-i}}{y^{i}} + \frac{a_{-1}}{y} + Q(y)
\end{equation}
for some $b_{-i} \in \mathbb{R}$ and some polynomial $Q(y)$. In particular, if $P(y) = (a_{-k})^{1/(1-k)}y$, then a straightforward calculation shows that $b_{-k} = 1$ in the expression above. However, we wish to find a polynomial $P$ such that not only does $b_{-k} = 1$ in the expression above, but $(b_{-k}, b_{-k+1}, \dots, b_2) = (1, 0, \dots, 0)$. 
The remainder of the proof will be inductive: assume that we can pick $P = \sum p_iy^i$ so that $P(0) = 0$, $P'(0) > 0$, and $(b_{-k}, b_{-k+1}, \dots, b_{-k+j-1}) = (1, 0, \dots, 0)$ in Equation \ref{eqn_bees} -- we aim to find a new $P$ so that $P(0) = 0$, $P'(0) > 0$, $(b_{-k}, b_{-k+1}, \dots, b_{-k+j}) = (1, 0, \dots, 0)$. 
For $t \in \mathbb{R}$ let $\widetilde{P} = P + tP^{j+1}$, we have for some smooth function $g$,
\begin{align*}
\sum_{i = 1}^{k} a_{-i} \frac{\widetilde{P}'(y)}{\widetilde{P}^i} 
&= \sum_{i = 1}^{k} a_{-i} \frac{P'}{P^i}\frac{(1 + (j+1)tP^{j})}{(1 + tP^j)^i}\\
&= \sum_{i = 1}^{k} a_{-i} \frac{P'}{P^i}\left( 1 + (j + 1 - i)tp_1^jy^j + gy^{j+1} \right) \\
&= \frac{1}{y^k} + \sum_{i = 2}^{k-j} \frac{b_{-i}}{y^{i}} + \frac{a_{-1}}{y} + Q(y)\\
& \ \ \ \ \ \  + \sum_{i = 1}^{k} a_{-i} \frac{P'}{P^i}\left((j + 1 - i)tp_1^jy^j + gy^{j+1}  \right) \\
\end{align*}
Notice that the $y^{-k+j}$ term of the above expression has coefficient
\[
b_{-k+j} + a_{-k} p_1^{1 - k}(j + 1 - k)tp_1^j = 0
\]
if we set $t = -b_{-k+j}p_1^{k - j - 1} / (a_{-k}(j + 1 - k))$, the $y^{-k+j}$ term vanishes, completing the induction.
\end{proof}

\end{document}